\documentclass{amsart}
\usepackage{amssymb,epsfig,amsmath}
\usepackage{lipsum}
\usepackage{enumitem}
\usepackage{hyperref}
\usepackage[all]{xy}
\usepackage[english]{babel}
\usepackage{amsfonts}
\usepackage{graphicx}
\usepackage{amsthm}
\usepackage{xcolor}
\usepackage{xcolor}
\usepackage{color}
\newcommand\Cov{\mathrm {Cov}}

\newtheorem{theorem}{Theorem}[section]
\newtheorem{theo}[theorem]{Theorem}

\newtheorem {coro}[theorem]{Corollary}
\newtheorem {pro}[theorem]{Proposition}
\newtheorem{defi}[theorem]{Definition}

\newtheorem{example}[theorem]{Example}
\newtheorem{lem}[theorem]{Lemma}
\newtheorem{remark}[theorem]{Remark}
\newtheorem{definition}[theorem]{Definition}

\newdimen\AAdi%
\newbox\AAbo%
\def\AArm{\fam0 }%\tenrm}%
\def\AAk#1#2{\setbox\AAbo=\hbox{#2}\AAdi=\wd\AAbo\kern#1\AAdi{}}%
\def\AAr#1#2#3{\setbox\AAbo=\hbox{#2}\AAdi=\ht\AAbo\raise#1\AAdi\hbox{#3}}%
\def\BBe{{\AArm I\!E}}%
\def\BBl{{\AArm I\!L}}%
\def\BBm{{\AArm I\!M}}%
\def\BBn{{\AArm I\!N}}%
\def\BBp{{\AArm I\!P}}%
\def\BBr{{\AArm I\!R}}%
\def\BBt{{\AArm T\AAk{-.62}{T}T}}%
\def\bbz{{\AArm Z\!\!Z}}%
\def\BBone{{\AArm 1\AAk{-.8}{I}I}}%

\RequirePackage[OT1]{fontenc}
\usepackage{hyperref}
\newdimen\AAdi%
\newbox\AAbo%
\def\AArm{\fam0 }%\tenrm}%
\def\AAk#1#2{\setbox\AAbo=\hbox{#2}\AAdi=\wd\AAbo\kern#1\AAdi{}}%
\def\AAr#1#2#3{\setbox\AAbo=\hbox{#2}\AAdi=\ht\AAbo\raise#1\AAdi\hbox{#3}}%
\def\BBn{{\AArm I\!N}}%
\def\BBe{{\AArm I\!E}}
\def\BBr{{\AArm I\!R}}
\def\BBp{{\AArm I\!P}}
\def\BBone{{\AArm 1\AAk{-.8}{I}I}}%
\def\lrar{\longrightarrow}

\def\bbz{{\mathbb Z}}
\def\bbx{{\mathbb X}}
\def\bbr{{\mathbb R}}

\def\bby{{\mathbb Y}}

\title[PR Supports of Stationary Dependent Random Variables] 
{Topological reconstruction of compact supports of dependent stationary random variables}
\begin{document}
\author{Sadok Kallel and Sana Louhichi}\thanks{The first author is supported by an SFRG-grant from the American University of Sharjah (AUS, UAE). The second author is supported by Univ. Grenoble Alpes, CNRS,  LJK}

\address{Sadok Kallel: American University of Sharjah, UAE, and Laboratoire Painlev\'e, Universit\'e de Lille, France.}
\email{sadok.kallel@univ-lille.fr}
\address{Sana Louhichi: 
Univ. Grenoble Alpes, CNRS, Grenoble INP*, LJK 38000 Grenoble, France.
*Institut of Engineering Univ. Grenoble Alpes
\\
700 Avenue Centrale,
38401 Saint-Martin-d'H\`eres, France.}
\email{sana.louhichi@univ-grenoble-alpes.fr}

\maketitle

\begin{abstract}
    In this paper we extend results on reconstruction of probabilistic supports of random i.i.d variables to supports of dependent stationary $\bbr^d$-valued random variables. All supports are assumed to be compact of positive reach in  Euclidean space. Our main results involve the study of the convergence in the Hausdorff sense of a cloud of stationary dependent random vectors to their common support. A novel topological reconstruction result is stated, and a  number of illustrative examples are presented. The example of the M\"{o}bius Markov chain on the circle is treated at the end with simulations.
\end{abstract}
\vspace{1cm}

%\tableofcontents

\section{Introduction}

Given a sequence of stationary random variables of unknown common law  and unknown  compact support $\BBm$ (Section \ref{results}), uncovering topological properties of $\BBm$ based on the observed data can be very useful in practice. Data analysis in high-dimensional spaces with a probabilistic point of view was initiated in \cite{smale} where data was assumed to be drawn from sampling an i.i.d. probability distribution or near a submanifold $\BBm$ of Euclidean space. Topological properties of $\BBm$ (homotopy type and homology) were deduced based on the random samples and the  geometrical properties of $\BBm$. Several papers on probability and topological inference ensued, some taking a persistence homology approach by providing a confidence set for persistence diagrams corresponding to the Hausdorff distance of a sample from a distribution supported on $\BBm$ \cite{Fasy}.

Topology intervenes in probability through reconstruction results (see \cite{attali, Ch, oudot, Fasy, kimetal, smale} and references therein). This research direction is now recognized as part of ``manifold learning''.
Given an $n$ point-cloud $\bbx_n$ lying in a support $\BBm$, which is generally assumed to be a compact subspace of $\bbr^d$ for some $d>0$, and given a certain probability distribution of these $n$ points on $\BBm$, one can formulate practical conditions from this data to reconstruct, up to homotopy or up to homology, this support $\BBm$. 
Reconstruction up to homotopy means recovering the homotopy type of $\BBm$. Reconstruction up to homology means determining, up to a certain degree, the homology groups of $\BBm$. Recovering the geometry of $\BBm$, including curvature and volume, is a much more delicate task  (see \cite{aamari, divol, fefferman, smale, wang}).

The goal of our work is to extend work of Nigoyi, Smale and Weinberger \cite{smale} from data drawn from sampling an i.i.d probability distribution that has support on 
a smooth submanifold $\BBm$ of Euclidean space to data drawn from stationary \textit{dependent} random variables concentrated inside a compact space of \textit{positive reach} (or PR-set). It is fitting here to define this notion: the \textit{reach} of a closed set $S$ in a metric space is the supremum $\tau\geq 0$ such that any point within distance less than $\tau$ of $S$ has a unique nearest point in $S$. Spaces of positive reach $\tau$ have been introduced in \cite{federer}, and form a natural family of spaces more general than convex sets $(\tau=\infty$) or smooth submanifolds, but sharing many to their common integro-geometric properties like ``curvature measures'' \cite{thale} (see Section \ref{reconstruct}). %All closed $C^2$ submanifolds of $\bbr^n$ have positive reach, the $C^1$ case being quite peculiar Section \ref{reconstruct}.

The interest in going beyond independence lies in the fact that many of the observations of everyday life are dependent, and independence is not sufficient to describe these phenomena. The study of the data support topologically and geometrically in this case can be instrumental in directional statistics for example, where the observations are often correlated. This can help get information on animal migration paths or
wind directions for instance. Modeling by a Markov chain on an unknown compact manifold, with or without boundary, makes it possible to study such examples.
Other illustrative examples can be found in more applied fields, for instance in cosmology \cite{cosmos}, medicine \cite{medecine}, imaging \cite{image}, biology \cite{biology}, environmental science \cite{en}, etc. 

To get information on an unknown support from stationary dependent data, we need to study the convergence in the Hausdorff distance $d_H$ of the data, seen as a (finite) point-cloud, to its support, similar to what was done in the i.i.d case \cite{Ch,CR2004, CR2009, Fasy}. The main interest in the metric $d_H$ is the following property: if $S\subset M$ is this point cloud in $M$, then $d_H(S,M)\leq \epsilon$ is equivalent to $S$ being $\epsilon$-dense in $M$ (see Section \ref{reconstruct}). We can expand this relationship to the random case as follows. 

\begin{definition}\label{densite}\rm We say that a point-cloud $\bbx_n$ of $n$ stationary dependent $\bbr^d$-valued random variables is $(\epsilon,\alpha)$-dense in $\BBm\subset\bbr^d$, for 
given $\epsilon> 0$ and $\alpha\in ]0,1[$,  if \begin{equation*}
\BBp\left(d_H(\bbx_n, \BBm)\leq \epsilon\right)\geq 1-\alpha.
\end{equation*}
If $\bbx:=(X_i)_{i\in \BBt}$ {{, with $\BBt$ being $\bbz, \BBn$ or $\BBn\setminus \{0\}$,}} is a stationary sequence of $\BBr^d$-valued random variables, we say that $\bbx$ is asymptotically $(\epsilon, \alpha)$-dense in $\BBm\subset\bbr^d$, with threshold $n_0$, {{ if for all positive $\epsilon$ sufficiently small and any $0<\alpha <1$,}}  there exists  a positive integer $n_0$ so that $\forall\ n\geq n_0$,
$\bbx_n=\{X_1,\ldots, X_n\}$ is $(\epsilon,\alpha)$-dense in $\BBm$. 
\end{definition}

The first undertaking of the paper is to identify sequences of dependent random vectors which are asymptotically $(\epsilon,\alpha)$-dense in a compact support. In Section \ref{examples} and Section \ref{sectionMarkov} we treat explicitly a number of examples and show for all of these that the property of being asymptotically $(\epsilon,\alpha)$-dense in the compact support holds by means of a key technical Proposition \ref{pro1} which uses blocking techniques to give upper bounds for $\BBp\left(d_H(\bbx_n , {\BBm}) > \epsilon\right)$.
Blocking techniques are very useful in the theory  of limit theorems for stationary
dependent random variables, and the idea behind is to view and manipulate blocks as ``independent'' clusters of dependent variables.

We summarize our first set of results into one main theorem. Given
$\bbx:=(X_i)_{i\in \BBt}$ a stationary sequence of $\BBr^d$-valued random variables, we write
$\rho_m(\epsilon)$ the concentration quantity of the block
$(X_1,\ldots, X_m)$, %(see \eqref{low1}).
{{that is, for $\epsilon>0$,
$$
\rho_m(\epsilon):=  \inf_{x\in \BBm_{dm}}\BBp( \|(X_1,\cdots,X_{m})^t-x\|\leq \epsilon),
$$
where $\BBm_{dm}$ denotes the support of the vector transpose $(X_1,\cdots,X_{m})^t$. }}
\begin{theo}\label{maintheo}
 The following stationary sequences of $\BBr^d$-valued random variables are asymptotically $(\epsilon,\alpha)$-dense in their common compact support:
\begin{enumerate}
\item Stationary $m$-dependent sequences
such that for any $\epsilon>0$, there exists a strictly positive constant $\kappa_{\epsilon}$ such that,
$\displaystyle
\rho_{m+1}(\epsilon) \geq \kappa_{\epsilon},
$ (Proposition \ref{coro2}). 
\item {{Stationary $m$-Approximable random
variables on a compact set. These are stationary models that can be approximated by $m$-dependent stationary sequences (see paragraph \ref{sectApp} and Proposition \ref{proapp}). }}
\item Stationary $\beta$-mixing sequences, with $(\beta_n)_n$ coefficients {{(see (\ref{betabelow}) for their definition)}},
such that for some $\beta>1$, and any $\epsilon>0$,
$$\lim_{m\rightarrow \infty} \rho_m(\epsilon)\frac{e^{m^{\beta}}}{m^{1+\beta}}=\infty,\,\,\,
{\mbox{and}}\,\,\,
\lim_{m\rightarrow\infty} \frac{e^{2m^{\beta}}}{m^2}\beta_m=0.
$$ (Proposition \ref{beta}).
 \item Stationary weakly dependent sequences, with $(\Psi(n))_n$  dependent coefficients (as introduced in (\ref{psidep})),
such that for some $\beta>1$, and any $\epsilon>0$,
$$\lim_{m\rightarrow \infty} \rho_m(\epsilon)\frac{e^{m^{\beta}}}{m^{1+\beta}}=\infty,\,\,\,
{\mbox{and that}}\,\,\,
\lim_{m\rightarrow\infty} \frac{e^{2m^{\beta}}}{m^2}\Psi(m)=0.
$$
 (Proposition \ref{psi}).
\item Stationary Markov chains
 with an invariant measure $\mu$ and a suitable transition probability kernel (assumptions (${\mathcal A}_1)$, (${\mathcal A}_2)$ of Section \ref{sectionMarkov}). See Proposition \ref{coro1} and
Proposition \ref{mobius}.
\end{enumerate}
\end{theo}

 Furthermore, and for each sequence $\bbx$ of random variables listed in Theorem \ref{maintheo}, methods for finding a threshold $n_0$, with sometimes explicit formulae for it, are given so that $\bbx_n$ is $(\epsilon,\alpha)$-dense in the common support, for $n\geq n_0$. 

The next step is topological and consists in showing that when the Hausdorff distance between $\bbx_n$ and the support is sufficiently small, it is possible to reconstruct the support up to homotopy.  We write $B(x,r)$ the closed ball in the Euclidean metric centered at $x$ with radius $r>0$, and we write $X\simeq  Y$ to mean that $X$ deformation retracts onto $Y$ with $Y\subset X$ (more precisely, this means that the identity map of $X$ is homotopic to a retraction onto $Y$, leaving $Y$ fixed during the homotopy).
%The following is Proposition \ref{mainrecon} in the text.

\begin{theo}\label{main2} 
Let $(X_i)_{i\in \BBt}$ be a stationary sequence of  $\BBr^d$-valued random variables with compact support $\BBm$ having positive reach $\tau$. 
Let $\epsilon\in \left(0, {\tau\over 2}\right)$, $r\in \left(\epsilon, {\tau\over 2}\right)$ and suppose that
$\bbx_n$ is $({\epsilon\over 2},\alpha)$-dense in $\BBm$. Then 
$$
\BBp\left(\displaystyle\bigcup_{x\in\bbx_n}B(x,r)\simeq  \BBm\right) \geq 1-\alpha.
$$
\end{theo}

The proof of this theorem is an immediate consequence of 
Definition \ref{densite} and a key reconstruction result proven in Section \ref{reconstruct} (Theorem \ref{mainrecon}) which gives the same minimal conditions for recovering the homotopy type of the support $\BBm$ from a sample of points $\bbx_n$ in $\BBm$. Theorem \ref{mainrecon} is ``deterministic'' and should have  wider applications. The key geometric ideas behind this result are in \cite{smale}, and in its extension in \cite{wang}, as applied to the approximation of Riemannian submanifolds. To get Theorem \ref{main2}, we weaken the regularity condition on the submanifolds from smooth to $C^{1,1}$, and in the hypersurface case we strengthen the bounds on the reach. This is then applied to thickenings of a positive reach set $M$ (see Section \ref{reconstruct}). It is important to contrast this result with earlier results in \cite{kimetal} (especially Theorem 19). There, the radii of the balls can be different. However, our theorem \ref{main2} is simpler to state and it is  easier to apply. 

Having stated our main results, which are mainly of probabilistic and topological interest, we can say a few words about the statistical implications. In practice, the point-cloud data are realizations of random variables living in unknown support $\BBm\subset\bbr^d$. We then ask to know if this support is a circle, a sphere, a torus, or a more complicated object. By taking sufficiently many points $\bbx_n$, our results tell us that the homology of $\BBm$ is the same as the homology of the union of balls around the data $\bigcup_{x\in\bbx_n} B(x,r)$, and this can be computed in general. The uniform radius $r$ depends on $\BBm$ only through its reach, which is then the only quantity we need to estimate or to know a priori. Knowing the homology rules out many geometries for $\BBm$.  Note that one may want to find ways to distinguish between a support that is a circle and one that is an annulus. However, conclusions of this sort are beyond the techniques of this paper.

\subsection{Contents} We now give some more details about the content of the paper and how it is organized. 
We start by establishing in Section \ref{reconstruct}  our homotopy reconstruction result of a support from a point cloud in the deterministic case. Everything afterward is of probabilistic nature, whereby point clouds are drawn from stationary random variables. In Section \ref{results} and Section \ref{examples} we state sufficient conditions in obtaining the asymptotically $(\epsilon,\alpha)$-dense property, that is conditions on concentrations and dependence coefficients under which $d_H(\bbx_n , {\BBm}) \leq \epsilon$ with large probability and for $n$ large enough.  More precisely, in Section \ref{results} we give general upper bounds for $d_H(\bbx_n,\BBm)$ using blocking techniques, i.e. by grouping the point cloud $\bbx_n$ into $k_n$ blocks, each block with $r_n$ points being considered as a single point in the appropriate Euclidean space of higher dimension.
This is stated in Proposition \ref{pro1}, which is the key result of this paper, where the control of $d_H(\bbx_n , {\BBm})$ is  reduced to the behavior of lower bounds of the concentration quantity of one block
\begin{equation}\label{low1}
\rho_{r_n}(\epsilon) = \inf_{x\in \BBm_{d{r_n}}}\BBp( \|(X_1,\cdots,X_{r_n})^t-x\|\leq \epsilon),
\end{equation}
and of
\begin{equation}\label{low2}
\inf_{x\in \BBm_{d{r_n}}}\BBp(\min_{1\leq i\leq k_n} \|(X_{(i-1)r_n+1},\cdots,X_{ir_n})^t-x\|\leq \epsilon),
\end{equation}
where, as before, $\BBm_{d{r_n}}$ is the support of the block  $(X_1,\cdots,X_{r_n})^t$. 
Clearly, for independent random variables,  a lower bound for (\ref{low1}) is directly connected to a lower bound for (\ref{low2}),  but this is not the case for dependent random variables, and we need to control (\ref{low1}) and (\ref{low2}) separately. Section \ref{examples} gives our main examples of stationary sequences of $\bbr^d$-valued random variables having good convergence properties, under the Hausdorff metric, to the support. For each example we check that conditions needed for the control of (\ref{low1}) and (\ref{low2}) can be reduced to conditions on  the concentration quantity $\rho_m(\epsilon)$ associated to the vector $(X_1,\cdots,X_m)^t$, for some fixed number of components $m\in \BBn\setminus\{0\}$.  In particular, for mixing sequences, the control of $d_H(\bbx_n,\BBm)$ is based on assumptions on the behavior of some lower bounds for this concentration quantity $\rho_m(\epsilon)$
in connection with the decay of the mixing dependence coefficients, as illustrated in Theorem \ref{maintheo}.
These lower bounds can be obtained by means of a condition similar to the so-called $(a,b)$-standard assumption
(see for instance \cite{Ch, CR2004, CR2009}) used in the case of i.i.d. sequences (i.e. when $k_n=n$ and $r_n=1$). However, our results in Section \ref{examples} generalize the case of i.i.d without assuming the $(a,b)$-standard assumption (Subsection \ref{rem}). 

Section \ref{sectionMarkov}  gives explicit illustrations of our main results and techniques in the case of stationary Markov chains. For this model, the quantities in (\ref{low1}) and (\ref{low2}) can be controlled from the behavior of a positive measure $\nu$ defining the transition probability kernel of this Markov chain, in particular from the lower bounds of the concentration quantity 
$\nu(B(x,\epsilon)\cap \BBm)$, for small $\epsilon$ and for $x\in \BBm$. The threshold $n_0$ can also be determined explicitly.
As a main illustration, the M\"obius Markov chain on the circle is studied in Subsection \ref{sectionmobius}, where $\BBm$ is the unit circle and $\nu$
is the arc length measure on the unit circle. The conditions leading to a suitable control of (\ref{low1}) and (\ref{low2}) are checked with no further assumptions and the threshold $n_0$ is   computed.

Section \ref{simu} gives an explicit simulation of a M\"obius Markov chain studied in \cite{Kato}. The intent here is to illustrate both the topological and probabilistic parts in an explicit situation. The simulation outcomes (Figures \ref{fig1} and \ref{fig2}) are in agreement with the theoretical results thus obtained. Finally, all deferred proofs appear in Section \ref{proofs}.

\vskip 5pt
\noindent{\sc Acknowledgements}: %We are very grateful to both referees for their insightful comments, and for suggesting important improvements. 
The first author would like to thank Sebastian Scholtes for insightful discussions on the material of Section \ref{reconstruct} and \cite{scholtes}. The second author is grateful to Sophie Lemaire for the present form of the proof of Lemma \ref{lem2}.

%%%%%%%%%%%%%%%%%%%%%%%%%%%%%%%%%%%%%%%%%%%%%%%%%%%%%%%%%%%%%%%%%

\section{ A reconstruction result}\label{reconstruct}

Given a point-cloud $S_n=\{x_1,\ldots, x_n\}$ on a metric space $M$, a standard problem is to reconstruct this space from the given distribution of points as $n$ gets large (see Introduction). Various reconstruction processes in the literature are based  on the Nerve theorem. This basic but foundational result can be found in introductory books in algebraic topology (\cite{hatcher}, chapter 4) and in most papers in manifold learning. This section takes a different route. 

Let's write below $B(x,r)$ (resp. $\mathring{B}(x,r)$) for the closed (resp. open) ball of radius $r$, centered at $x$. Starting with a point-cloud $S_n=\{x_1,\ldots, x_n\}\subset M$, with $M$ a compact subset of $\bbr^d$ with its Euclidean metric $\|\cdot\|$, we therefore seek conditions on some radius $r$ and on the distribution of the points of $S_n$ so that the union of balls $\bigcup_{i=1}^nB(x_i,r)$ deformation retracts onto $M$.
The $r$-\textit{offset} (or $r$-\textit{thickening} or $r$-\textit{dilation} or $r$-\textit{parallel set} depending on the literature) of a closed set $M$ is defined to be
$$M^{\oplus r} := \{p\in\bbr^d\ |\ d(p, M):=\inf_{x\in M}||x-p|| \leq r\}= \bigcup_{x\in M} B(x,r)$$
Many of the existing theorems in homotopic and homological inference are about offsets. In terms of those, the Hausdorff distance $d_H$ between two \textit{closed} sets $A$ and $B$, is defined to be
\begin{equation}\label{hd}
d_H(A,B) = \inf_{r>0}\{A\subset B^{\oplus r}, B\subset A^{\oplus r}\}=\max\left(\sup_{x\in A}\inf_{y\in B}\|x-y\|,\,\, \sup_{x\in B}\inf_{y\in A}\|x-y\|\right)
\end{equation}
(replacing $\inf$ and $\sup$ with $\min$ and $\max$ for compact sets).
This is a ``coarse'' metric in the sense that two closed spaces $A$ and $B$ can be very different topologically and yet be arbitrarily close in Hausdorff distance. 

We say that a subset $S\subset M$ is  $\epsilon$-dense (resp. strictly $\epsilon$-dense) in $M$, for some $\epsilon>0$, if $B(p,\epsilon)\cap S \neq\emptyset$ 
(resp. $\mathring{B}(p,\epsilon)\cap S \neq\emptyset$ ) for each $p\in M$. We have the following characterization.

\begin{lem} Let $S\subset M$ be a closed subset. Then 
$$S \ \hbox{is}\ \epsilon\hbox{-dense in $M$} \Longleftrightarrow M\subset S^{\oplus\epsilon}\ \Longleftrightarrow\ d_H(S,M)\leq\epsilon$$
\end{lem}

\begin{proof}
When $S\subset M$,  
$d_H(S,M) = \inf\{r>0\ | M\subset S^{\oplus r}\}$.
If $S$ is $\epsilon$-dense, any $p$ in $M$ is within $\epsilon$ of an $x\in S$, and so
$M\subset S^{\oplus\epsilon}$, which implies that
$d_H(S,M)\leq\epsilon$. The converse is immediate.
%$d_H(S,M) = \sup_{p\in M}\inf_{x\in S}d(x,p)$, where $d$ is Euclidean distance. Saying that $d_H(S,M)\leq \epsilon$ means that for all $p\in M$, $\inf_{x\in S }d(x,p)=\min_{x\in S}d(x,p)\leq \epsilon$, and so there is a point in $S$ within distance $\epsilon$ from $p$. 
\end{proof}

%Scholtes https://arxiv.org/pdf/1304.4179v1.pdf

From now on, $S$ will always mean a point-cloud in $M$; that is a finite collection of points. The following is a foundational result in the theory, and is our starting point.

\begin{theorem}\label{smaleprop} (\cite{smale}, Proposition 3.1) Let $M$ be a compact Riemannian submanifold of $\bbr^d$ with positive reach $\tau$, and
$S \subset M$ a strictly $\frac{\epsilon}{2}$-dense finite subset for $\epsilon < \sqrt{\frac{3}{5}}\tau$. Then for any $ r\in [\epsilon , \sqrt{\frac{3}{5}}\tau[$,  $\bigcup_{x\in S }\mathring{B}(x,r)\simeq M$. 
\end{theorem}

\begin{remark}\rm Theorem \ref{smaleprop} is a topological ``reconstruction'' result which recovers the homotopy type of $M$ from a finite sample. There are many reconstruction methods in the literature that are too diverse to review here (see \cite{attali, divol, kimetal} and references therein). Reconstructions can be topological, meaning they recover the homotopy type or homology of the underlying manifold $M$, or they can be geometrical. We only address the topological aspect in this paper. In that regard, Corollary 10 in \cite{kimetal} is attractive for its simplicity as it proves a general reconstruction result for compact sets with positive reach by applying the nerve theorem to a cover by ``subspace balls'' $\mathcal U_M = \{B(x_i,r)\cap M\}$.  For Riemannian manifolds $M$, and thus in the $C^\infty$ case, there is an alternative intrinsic geometric method for homotopy reconstruction based on ``geodesic balls''. Let 
$\rho_c>0$ be a \textit{convexity radius} for $M$. Such a radius has the property that around each $p\in M$, there is a ``geodesic ball'' $B_g(p,\rho_c)$
which is convex, meaning that any two points in this neighborhood are joined by a unique geodesic in that neighborhood. These geodesic balls, and their non-empty intersections, are contractible. If $S_n=\{x_1,\ldots, x_n\}$ is a point-cloud such that
$ \{B_g(x_i,\rho_c)\}$ is a cover of $M$,  then $\bigcup_iB_g(x_i,\rho_c)\simeq M$ by the nerve theorem. 
\end{remark}
%%%%%%

\subsection{Positive reach} The notion of positive reach is foundational in convex geometry. As indicated in the introduction, the reach of a  subset $M$ is defined to be
\begin{equation}\label{defreach}
\tau (M) := \hbox{sup}\{r\geq 0\ |\ \forall y\in M^{\oplus r}\,\, \exists !\ x\in M\ \hbox{nearest to $y$}\}
\end{equation}

A PR-set is any set $M$ with $\tau (M)> 0$. Compact submanifolds are PR. Figure \ref{positive} gives an example of a PR-set that is not a submanifold. The quintessential property of PR-sets is the existence, for $0<r<\tau$, of the ``unique closest point'' projection
\begin{equation}\label{projection}
\pi_M : M^{\oplus r}\lrar M\ \ ,\ \ 
||y - \pi_M(y)|| = d_H(y,M)
\end{equation}
with $\pi_M(y)$ the  unique nearest point to $y$ in $M$.
PR-sets are necessarily closed, thus  compact if bounded. 
\begin{figure}[htb]
\begin{center}
\epsfig{file=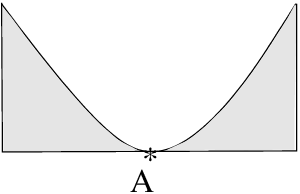,height=0.8in,width=1.5in,angle=0.0}
\caption{This space has positive reach $\tau$ in $\bbr^2$ but a neighborhood of point $A$ indicates it is not a  submanifold (with boundary). 
}\label{positive}
\end{center}
\end{figure}

 As already indicated, we use the notation $X\simeq  Y$, if $Y\subset X$, to denote the fact that $X$ deformation retracts onto $Y$. ``Thin enough'' offsets of PR-sets deformation retract onto $M$.

\begin{lem}\label{retract}
    Let $M$ be a PR-set with $\tau=\tau (M)>0$. Then $M^{\oplus r}\simeq M$, whenever $r<\tau$.
\end{lem}

\begin{proof} This is immediate once we see that if $p\in M$ and $x\in \pi^{-1}_M(p)\subset M^{\oplus r}$, then the entire segment $[x,p]$ of $\bbr^d$ must be in $\pi_M^{-1}(p)$, so we can use the homotopy 
$F: M^{\oplus r}\times [0,1]\rightarrow M^{\oplus r}$,
$F(x,t)=(1-t)x+t\pi_M(x)$, $t\in [0,1]$, to get a (linear) deformation retraction, with $M$ being fixed during the homotopy.
\end{proof}

The original interest in sets of positive reach is that they have suitable small parallel neighborhoods with no self-intersections which allow one to compute their volume. This leads to a Steiner-type formula and a definition of curvature measures for these sets (see \cite{thale}). If $M$ is a compact  Riemannian submanifold in $\bbr^d$, as considered in \cite{smale}, then $\tau (M)$ is positive and is the largest number having the property that the open normal bundle about
$M$ of radius $r$ is smoothly embedded in $\bbr^d$ for every $r < \tau$. It is enough however for $M$ to be $C^2$ to ensure that $\tau (M)>0$ (\cite{thale}, Proposition 14), and that it is even enough to be $C^{1,1}$ in case $M$ is a closed hypersurface (see \cite{scholtes}), Theorem 1.3, which is an ``if and only if'' statement). We recall the definition of $C^{1,1}$ (\cite{hormander}, Def. 2.4.2.).

\begin{definition}
 A manifold $M\subset\bbr^d$ is said to have  $C^{1,1}$ boundary $\partial M$ if on a local neighborhood $U$ of $x_0\in \partial M$,
 $U\cap M = \{x\in M\ |\ x_1\geq f(x')\}$, where
 $x'=(x_2,\ldots, x_d)$, $f$ is $C^1$ and grad$(f)$ is Lipshitz continuous.
\end{definition}

What is crucial to us in this section are the next two results. For general discussion, we  refer to \cite{federer, fu} for Proposition \ref{riemannian}, and \cite{barb, scholtes} for Proposition \ref{normal} next. Throughout, a manifold is assumed to be compact, without boundary unless we specify the contrary. For $x\in\bbr^d$, let $d_M(x )=d(x , M ) = inf \{d(x,y), y \in M\}$ be the distance function to $M$. This function is $1$-Lipshitz, and it is continuously differentiable when restricted to the interior of $M^{\oplus r}\setminus M$ if $r<\tau$ (\cite{federer}, Theorem 4.8). Elementary pointset topology shows that the interior of the $r$-offset of $M$ is $int(M^{\oplus r}) = \bigcup_{x\in M}\mathring{B}(x,r) = d^{-1}_M[0,r)$, and the topological boundary is $d_M^{-1}(r)$. 

\begin{pro}\label{riemannian}\cite{federer} Let $M\subset\bbr^d$ be compact of positive reach $\tau$. For $0<r <\tau$, $M^{\oplus r}$ is a compact manifold with $C^{1,1}$-boundary.
\end{pro}

Before stating the next Proposition, we make a definition: we say that a submanifold $M\subset\bbr^d$ has a tubular neighborhood ``foliated by orthogonal disks'' if it has a tubular neighborhood $T$ (i.e. an embedding of the normal bundle extending the embedding of $M$) with a well-defined continuous (unique) nearest point projection $\pi_M: T\rightarrow M$.

\begin{pro}\label{normal} 
 A closed submanifold $N$ in $\bbr^d$, $d\geq 2$, has a tubular neighborhood ``foliated by orthogonal disks'' if and only if it is $C^{1,1}$.
 %see thale section 2.2
\end{pro}

\begin{proof} This is a consequence of Theorem 1.3 of \cite{scholtes} which proves that $N$ is $C^{1,1}$ if and only if it has positive reach $\tau$. The tubular neighborhood $T$ is then all points a distance stricly less than $\tau$ from $N$. 
\end{proof}

\begin{remark}\rm A very informative discussion about the above is on MO \cite{mo}, and the point is this. In the $C^1$-case, the choice of the (unit, outer) normal vector at every point of $N$ is a continuous function (this is by definition the Gauss map). In fact if $N$ is $C^k$, then the choice of a normal $N\rightarrow\bbr^n$ is $C^{k-1}$ (see \cite{barb}, Lemma 4.6.18). If we have $C^1$-regularity but not $C^{1,1}$, it could happen that the normals intersect arbitrarily close to the hypersurface, in which case the reach is $0$ indeed. A good example to keep in mind, and which we owe to S. Scholtes\footnote{Private communication.}, is the graph of the real-valued function which is $0$ for $x\leq 0$ and $x^{3/2}$ for $x\geq 0$. This function is $C^{1, {1/2}}$, not $C^{1,1}$, and one observes that near $0$, the normals intersect arbitrarily close to the curve. 
\end{remark}

If $M$ is a compact PR-set, its offset $M^{\oplus r}$ is also compact and PR for $r<\tau$, with reach $\tau-r$, where $\tau$ is the reach of $M$.  This assertion is not entirely evident, since generally, the reach is not always well-behaved for nested compact sets. By this we mean that if $(K_2,K_1)$ is a pair of nested compact sets in $\bbr^d$, $K_1\subset K_2$, then both cases $\tau_1<\tau_2$ or $\tau_2<\tau_1$ can occur, where $\tau_i$ is the reach of $K_i$. For example and in the former case, take $K_1$ to be the circle and $K_2$ to be the closed disk, while for the latter case, take $K_1$ to be a point in a finite reach $K_2$. The case of $(K_2,K_1) = (M^{\oplus r}, M)$ is therefore special.

\begin{lem}\label{reachoffset} Let $M\subset\bbr^d$ be a compact PR-set with reach $\tau$, and $0\leq r<\tau$, then $M^{\oplus r}$ has positive reach with $\tau (M^{\oplus r})=\tau-r>0$. \end{lem}

\begin{proof} Essentially, the point is that any ray from $y\not\in M^{\oplus r}$ to $M$ must cut the boundary $\partial M^{\oplus r}$ at a point a distance $r$ to $M$.
Suppose $r>0$, so that $M^{\oplus r}$ is a codimension $0$-manifold with boundary $\partial M^{\oplus r}$ in $\bbr^d$. Write $\tau_r$ its reach. We will first prove that if $y$ in the complement of $M^{\oplus r}$ has a unique projection onto $M$, then necessarily it has a unique projection onto $M^{\oplus r}$ (this will prove that $\tau_r>0$ and that $\tau - r\leq \tau_r$). Reciprocally, we will argue that is if $y$ has a unique projection onto $M$, then it also has a unique projection onto $M^{\oplus r}$.

To prove the first claim,  write  $y_1 = [y,\pi_M(y)]\cap\partial M^{\oplus r}$. We claim that $y_1$ is the unique closest point to $y$ in $M^{\oplus r}$.  Indeed if there is $z_1$ on that boundary that is closer to $y$, then
$$d(y,\pi_M(z_1))\leq d(y,z_1) + d(z_1,\pi_M(z_1)) = d(y,z_1)+r
\leq d(y,y_1) + d(y_1,M)
= d(y,\pi_M(y))$$
and so, $d(y,\pi_M(z_1))=d(y,\pi_M(y))$ (since $d(y,\pi_M(y))$ is smallest distance of $y$ to $M$), and by uniqueness, $\pi_M(z_1)=\pi_M(y)$. This implies that $d(y,z_1) + d(z_1,M)=d(y,M)$, and so $y,z_1,\pi_M(y)$ are aligned. This can only happen if $y_1=z_1$.

Suppose now that $y$ has a unique projection onto $M^{\oplus r}$ which we label $y'$. We can check that it also has a unique projection onto $M$.  Let $z$ be that projection. By a similar argument as previous, $z$ must be $\pi_M(y')$ (so unique) and $y,y',z$ are aligned. This shows reciprocally that $\tau_r+r\leq \tau$.

The above arguments show that $\tau = \tau_r+r$, and in fact they can be used to show that
$$M^{\oplus r'}= (M^{\oplus r})^{\oplus (r'-r)} $$
for all $r\leq r'< \tau$.
\end{proof}

\subsection{Manifolds with boundary}\label{hausdorff}
In order to apply our ideas to PR sets, we need to extend Theorem \ref{smaleprop} from closed Riemannian submanifolds to submanifolds with boundary. Note that the reach of $\partial M$ (manifold boundary) and $M$ are not comparable in general. Indeed, take
$M$ to be the $y=\sin (x)$ curve on $[0,\pi]$, with boundary the endpoints. Then $\tau (M) < \tau (\partial M)$. Take now a closed disk $M$ in $\bbr^2$. Then $\tau (\partial M)<\tau (M)=\infty$. If $M$ is of codimension $0$, then $\tau (\partial M)\leq\tau (M)$ always.

In \cite{wang}, the authors managed to extend Theorem \ref{smaleprop} to smooth submanifolds with boundary, and showed in this case, that the bound $\sqrt{3\over 5}\tau$ in Theorem \ref{smaleprop} can be replaced by $\delta\over 2$, where $\delta = \min (\tau (M), \tau (\partial M))$. 
We revisit this result in the codimension $0$-case and refine it.  

\begin{pro}\label{codimension0} 
 Let $M$ be a compact codimension $0$-submanifold of $\bbr^d$ with $C^{1,1}$-boundary and having positive reach $\tau=\tau (M)>0$, and let
$S \subset M$ be an $\epsilon\over 2$-dense finite subset with $\epsilon < {\tau\over 2}$. Then for any $r$ such that $\epsilon\leq r<  {\tau\over 2}$,  $\bigcup_{x\in S} {B}(x,r)\simeq M$.
\end{pro}

Notice that $M$ need not be connected. Notice also that we have weakened the regularity on $\partial M$ from smooth to $C^{1,1}$. According to \cite{scholtes}, Theorem 1.3) (see also \cite{federer}, Remark 4.20), this condition is enough to ensure that $\tau (\partial M)>0$. Finally, notice  that we use closed balls in our statement, and that they may have larger radius than $\epsilon$, but not exceeding $\tau\over 2$.

\begin{proof}  The proof is an adaptation of Lemma 4.1 of \cite{smale} and Lemma 4.3 of \cite{wang} for smooth submanifolds. For completeness, we will reconstruct the part of the argument that we need.

Firstly, since the reach of the disjoint finite union of PR sets is the least of their reaches and their pairwise distances, we can assume without loss of generality that $M$ is connected from the start.
Let $M$ be of connected of codimension $0$. Then its boundary is a connected codimension $1$ closed submanifold (i.e. a closed hypersurface). It divides Euclidean space into two regions. 
We let $\tau^-$ denote the reach of a component of $\partial M$ in the region contained in $M$ (the interior region), and $\tau^+$ its reach within the other open region (exterior region). Clearly $\tau:=\tau (M)=\tau^+$.

Now $\partial M$ is $C^{1,1}$ by hypothesis, and being a hypersurface, it is necessarily orientable. It has a continuous normal vector field into the exterior region, defining a (trivial) subbundle $T(\partial M)^+$ of $T(\partial M)$. We write $T_p^{\perp,+}(\partial M)$ the fiber at $p\in\partial M$ which is a half line extending into the exterior region, perpendicular to $T_p(\partial M)$. Linear deformation retraction along this direction as in Lemma \ref{retract}, keeping $M$ fixed, shows that $M^{\oplus r}\simeq M$, as long as $r<\tau^+=\tau $ (where normal directions never intersect). We have that
$$M^{\oplus r} = \bigcup_{x\in M}B(x,r)\simeq M\ ,\ r<\tau $$
We want that this retraction of $M^{\oplus r}$ onto $M$ (along fibers of $T^+(\partial M)$) restricts to a deformation retraction of the middle space $W_S$, figuring in the sequence of inclusions below, onto $M$
$$M\subset W_S:=\bigcup_{x\in S}B(x,r)\subset M^{\oplus r}\ ,\ {\epsilon\over 2}\leq r<\tau $$
That is we only take the union of balls centered at points of $S$. This covers $M$ since $r\geq {\epsilon\over 2}$.
Let's see how the deformation retraction of the bigger space $M^{\oplus r}$ onto $M$ may fail to restrict to a retraction on $W_S$: let $v\in T_p^{\perp,+}(\partial M)$, and suppose $v\in B(q,r)$, with $q\in S$ but $q\not\in B(p,r)$. So the line segment $[v,p]$ is not in the ball $B(q,r)$, and the linear retraction will leave that ball eventually. This however will not cause a problem as long as the segment falls in another ball and does not get out of the entire union 
$\bigcup_{x\in S}B(x,r)$. This happens if both $v,p$ are in some other ball $B(x,r)$, $x\in S$ (because balls are convex). Now, any such $x$ is at most a distance $\epsilon\over 2$ from $p$, as illustrated in the figure below.
\begin{figure}[htb]
\begin{center}
\epsfig{file=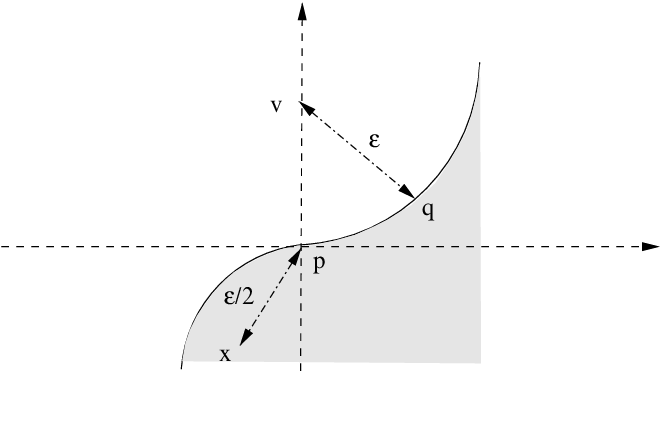,height=1.8in,width=2.3in,angle=0.0}
\caption{An extreme disposition of points, $x,q\in S $ and $p,q\in\partial M$. The points $q,p$ are on a circle tangent to $T_p(M)$, of radius $\tau$ and center on the vertical dashed line representing the normal direction $T_p^{\perp,+}(M)$, pointing in the exterior region, while $x$ is anywhere in $M\cap S$ at a distance at most ${\epsilon\over 2}$ from $p$.}\label{boundary}
\end{center}
\end{figure}

With the above target in mind, consider the following configuration of points: $p\in\partial M$, $v\in T_p^{\perp,+}(\partial M)\cap B(p,\tau)$, 
$q\in S$ and $v\in B(q,r)$ but
$p\not\in B(q,r)$. How far can $v$ be from $p$, among all choices of such points $q$? The answer can be extracted from the key Lemma 4.1 of \cite{smale}.
The ``worst case scenario'', corresponding to when $v$ would be further from $p$, is when $q$ and $p$ lie on the circle of radius $\tau$, with center in $T_p^{\perp,+}$ as in the figure (i.e. $p,q,v$ making up an isoceles triangle with two sides equal to $r$).  Lemma 4.1 of \cite{smale} applied to this situation gives that
$\displaystyle d(v,p)< {r^2\over \tau^+} = {r^2\over\tau} <\tau$.

Next, we look for an $x\in B(p,{\epsilon\over 2})\cap S\neq\emptyset$ which is within distance $r$ from $v$. The worst case scenario again (that is the furthest this $x$ can be from $v$) is when $v,p, x$ are aligned and $d(p,x)={\epsilon\over 2}$. Note that this does not happen if $\tau^-$ is very small compared to $\epsilon\over 2$ for example, but again, we only care about the extreme case. Combining with the previous inequality, we see that in all cases, for any such choice of $x$, necessarily $\displaystyle d(v,x)\leq {\epsilon\over 2} + {r^2\over \tau}$.  This is all illustrated in our Figure \ref{boundary} which is the analog of Fig.2 of \cite{smale} and Fig.1 of \cite{wang}. 
Since we want $v\in B(x,r)$, it is enough to require that 
\begin{equation}\label{restriction}
   {\epsilon\over 2} + {r^2\over \tau}< r
\end{equation}
or $r^2-r\tau + {\epsilon\tau\over 2}<0$. Any value of $r$ between the roots of this equation does the job, and it is immediate to see that $r\in \left[\epsilon, {\tau\over 2}\right[$ satisfies this condition. The proof is complete.

Note that in the case of Theorem \ref{smaleprop} of \cite{smale}, that is when one is considering $M$ that is closed (i.e. no boundary),  the point $x$ to be chosen couldn't be ``anywhere'' possibly around $p\in M$ as in Figure \ref{boundary}, but had to lie on $M$ (which would be the boundary in that figure), and thus the authors get a different bound on $r$. 
\end{proof}

We finally come to the proof of the main reconstruction result of this section, and of which Theorem \ref{main2} is a consequece\footnote{We thank the referee for suggesting this simpler statement than the one we originally gave.}. 

\begin{theorem}\label{mainrecon} Let $M$ be a compact space in $\bbr^d$ with positive reach $\tau$, let $\epsilon\in \left(0, {\tau\over 2}\right)$ and $r\in \left(\epsilon, {\tau\over 2}\right)$. If $S\subset M$ is ${\epsilon\over 2}$-dense, then $\displaystyle\bigcup_{x\in  S }B(x,r)\simeq  M$. 
\end{theorem}

\begin{proof} 
 Notice that by Proposition \ref{codimension0}, the result is true if $M$ is a codimension $0$ submanifold whose boundary is $C^{1,1}$ of reach $\tau>0$ (let's refer to this submanifold as a ``good object''). The idea now is very simple and relies on the fact that, if $M$ itself is not such a good object but is of positive reach, then any slight thickening of it will be a good object.

Suppose $S $ is an ${\epsilon'-\delta\over 2}
$-dense sample in $M$, $S \subset M$, then it is an 
$\epsilon'\over 2$-dense sample in $M^{\oplus \delta/2}$. This offset is a codimension $0$ submanifold of $\bbr^d$, with $C^{1,1}$ boundary, and containing $M$.
Its reach is $\tau'=\tau-{\delta\over 2}$ (Lemma \ref{reachoffset}). Proposition \ref{codimension0} implies then that
for all $\epsilon'\leq r< {\tau-\delta/2\over 2}$,
$\bigcup_{x\in S }B(x,r)\simeq   M^{\oplus\delta/2}$. But if we choose ${\delta\over 2}<\tau$,
$M^{\oplus\delta/2}$ deformation retracts onto $M$ (Lemma \ref{riemannian}), and  the composition of both deformation retractions gives a deformation retraction
$\bigcup_{x\in S }B(x,r)\simeq   M$.

Set $\epsilon = \epsilon'-\delta$.
What the above proves is that
 if $S $ 
is $\displaystyle{\epsilon\over 2}$-dense in $M$, then for any $r$ such that
\begin{equation}\label{inegalite}\epsilon+\delta\leq r< {\tau\over 2}-{\delta\over 4}
\end{equation}
where $\delta$ is any number such that
 $\displaystyle 0<\delta < {2\over 5}\tau$ and
$\displaystyle 0<\epsilon < {\tau\over 2}-{5\delta\over 4}$, we must have that
 $\displaystyle\bigcup_{x\in  S }B(x,r)\simeq  M$. Now, given $0<\epsilon<{\tau\over 2}$ and $\epsilon <r< {\tau\over 2}$ (as in the hypothesis), set
$$\delta = \min\left\{{1\over 2}\cdot {4\over 5}\left({\tau\over 2}-\epsilon\right), r-\epsilon, 4\left({\tau\over 2}-r\right)\right\}$$
Then by construction, and for that $\delta$, one can check that \eqref{inegalite} is satisfied, and hence the desired deformation retraction is ensured.
\end{proof}

%%%%%%%%%%%%%%%%%%%%%%%%%%%%%%%%%%%%%%%%%%%%%%%%%%%%%%

\section{Blocking techniques and upper bounds for the Hausdorff distance}\label{results}

This section states and proves the main technical result of this paper. It is given by Proposition 
\ref{pro1} below, which is general and of independent interest.  It is based on  blocking techniques as well as a useful geometrical result, proven in \cite{smale}, relating the minimal number of a covering of a compact set by closed balls to the maximal length of chains of points whose pairwise distances are bounded below.

Let $(X_i)_{i\in \BBt}$ (where $\BBt$ is either $\bbz$ or $\BBn$ or  $\BBn\setminus \{0\}$) be a stationary sequence of $\BBr^d$-valued random variables. Let $P$ be the distribution of $X_1$. Suppose that
$P$ is supported on a compact set $\BBm$ of $\BBr^d$, i.e. $\BBm := {\mbox{supp}} (X_j)$ is the smallest closed set carrying the mass of $P$;
\begin{equation}\label{supportdef}
\BBm=\bigcap_{C\subset \BBr^d,\ 
P(\overline{C})=1}\overline{C},
\end{equation}
where $\overline{C}$ means the closure of the set $C$ in Euclidean space.
Recall that $\bbx_n=\{X_1,\cdots,X_n\}$  and this is viewed as a 
subset of $\BBr^d$.
Throughout, we will be working with the Hausdorff distance $d_H$ \eqref{hd}. Note that $d_H(\{x\},\{y\})=||x-y||$ (Euclidean distance) if $x,y$ are points. Beware that the distance of a point $y$ to a closed set $A$ is $d(y,A)=\inf_{x\in A}||x-y||$, while its Hausdorff distance to $A$ is $d_H(y,A) = \sup_{x\in A}||x-y||$. This explains in part why this metric is very sensitive to outliers (see \cite{robust}) and to noisy phenomena.

We wish to give upper bounds for $\BBp\left(d_H(\bbx_n, \BBm)> \epsilon\right)$ via a blocking technique. Let $k$ and $r$ be two positive integers such that $kr\leq n$. Define, for $1\leq i\leq k$, the random vector $Y_{i,r}$ of $\BBr^{dr}$, by
$
Y_{i,r}=(X_{(i-1)r+1},\cdots,X_{ir})^t.
$
Let
$$
\bby_{k} =\{Y_{1,r},\cdots,Y_{k,r}\}
$$

{{be a subset of $\BBr^{dr}$ of stationary $k$ random vectors which are not necessarily independent. }}
The support $\BBm_{dr}$ of the vector  $Y_{1,r}$  is  included in $\BBm\times \cdots \times\BBm$ ($r$ times) and since, by definition, $\BBm_{dr}$ is a closed set, it is necessarily compact in $\BBr^{dr}$. As we now show, it is possible to reduce the behavior of
$d_H(\bbx_n , {\BBm})$ to that of the sequence of vectors
$(Y_{i,r})_{1\leq i\leq k}$ for any $k$ and $r$ for which 
$kr\leq n$ and under only the assumption of stationarity of
$(X_i)_{i\in \BBt}$. 

\begin{pro}\label{pro1} With $\epsilon>0$, $k$ and $r$ any positive integers such that $kr\leq n$, it holds that
\begin{eqnarray*}
&&\BBp\left(d_H(\bbx_n , {\BBm}) > \epsilon\right) 
\leq \BBp\left(d_H(\bby_{k} , {\BBm}_{dr}) > \epsilon\right) 
\leq  \frac{\sup_{x\in {\BBm}_{dr}}\BBp\left(\min_{1\leq i\leq k}\|Y_{i,r}-x\|> \epsilon/2\right)}{1-\sup_{x\in {\BBm}_{dr}}\BBp\left(\|Y_{1,r}-x\|> \epsilon/4\right)}.
\end{eqnarray*}
\end{pro}

\begin{proof}
Since $\BBp\left(\bby_{k}  \subset {\BBm}_{dr}\right)=1$, we have almost surely (a.s)
\begin{eqnarray}\label{n1}
&& d_H(\bby_{k} , {\BBm}_{dr}) = \sup_{x\in {\BBm}_{dr}}\min_{1\leq j\leq k}\|Y_{j,r}-x\|.
\end{eqnarray}
Since ${\BBm}_{dr}$ is compact, there exists a finite set ${\mathcal C}_N=\{c_{1},\cdots,c_{N}\}\subset {\BBm}_{dr}\subset {\BBr}^{dr}$ of centers of balls, forming a minimal $\epsilon$-covering set for ${\BBm}_{dr}$ so that, for a fixed $x\in {\BBm}_{dr}$, there exists $c_i \in {\mathcal C}_N\subset {\BBm}_{dr}$ such that
$$
\|x-c_i\|\leq  \epsilon.
$$
Hence,
$$
\|Y_{j,r}-x\| \leq \|Y_{j,r}-c_i\|+ \|c_i-x\|\leq \|Y_{j,r}-c_i\|+ \epsilon.
$$
Consequently, for any $x\in  {\BBm}_{dr}$,
$$
\min_{1\leq j\leq k}\|Y_{j,r}-x\| \leq \min_{1\leq j\leq k}\|Y_{j,r}-c_i\| +  \epsilon \leq \max_{1\leq i\leq N}\min_{1\leq j\leq k}\|Y_{j,r}-c_i\| +  \epsilon
$$
and
$$
\sup_{x\in {\BBm}_{dr}}\min_{1\leq j\leq k}\|Y_{j,r}-x\|\leq \max_{1\leq i\leq N}\min_{1\leq j\leq k}\|Y_{j,r}-c_i\| +  \epsilon.
$$
Hence,
\begin{eqnarray}\label{n2}
&& \BBp\left(\sup_{x\in {\BBm}_{dr}}\min_{1\leq j\leq k}\|Y_{j,r}-x\| \geq 2\epsilon\right) \leq \BBp\left( \max_{1\leq i\leq N}\min_{1\leq j\leq k}\|Y_{j,r}-c_i\| \geq \epsilon\right){\nonumber}\\
&& \leq N \max_{1\leq i\leq N}\BBp\left(\min_{1\leq j\leq k}\|Y_{j,r}-c_i\| \geq \epsilon\right)\leq N \sup_{x\in {\BBm}_{dr}}\BBp\left(\min_{1\leq j\leq k}\|Y_{j,r}-x\| \geq \epsilon\right).
\end{eqnarray}
We have now to bound $N$. For this we use Lemma 5.2 in \cite{smale}
(as was done in \cite{Fasy}), to get
\begin{equation}\label{smale1}
N\leq \left({\inf_{x\in \BBm_{rd}}\BBp(\|Y_{1,r} - x\| \leq  \epsilon/2)}\right)^{-1}= \left(1-{\sup_{x\in \BBm_{rd}}\BBp(\|Y_{1,r}-x\|> \epsilon/2)}\right)^{-1}.
\end{equation}
Hence, by (\ref{n1}) together with (\ref{n2}) and (\ref{smale1}),
\begin{eqnarray}\label{b3}
&& \BBp\left(d_H(\bby_{k} , {\BBm}_{dr}) > 2\epsilon\right)\\
&& \leq \left(1-{\sup_{x\in \BBm_{rd}}\BBp(\|Y_{1,r}-x\|> \epsilon/2)}\right)^{-1}\sup_{x\in \BBm_{rd}}\BBp\left(\min_{1\leq j\leq k}\|Y_{j,r}-x\| \geq \epsilon\right). {\nonumber}
\end{eqnarray}
Thanks to (\ref{b3}), the proof of this proposition is complete if we prove that,
\begin{equation}\label{pro13bis}
\BBp\left(d_H(\bbx_n , {\BBm}) > \epsilon\right)\leq \BBp\left(d_H(\bby_{k} , {\BBm}_{dr}) > \epsilon\right).
\end{equation}
Recall that $\BBp(\bbx_n \subset \BBm)=1$, so that 
$ d_H(\bbx_n , {\BBm}) = \sup_{x\in \BBm}\min_{1\leq j\leq n}\|X_j-x\|,
$
and, since $kr\leq n$,
$$
d_H(\bbx_n , {\BBm})=\sup_{x\in \BBm}\min_{1\leq j\leq n}\|X_j-x\|\leq \sup_{x\in \BBm}\min_{1\leq j\leq kr}\|X_j-x\|=d_H(\bbx_{kr} , {\BBm}).
$$
From this we deduce that
\begin{equation}\label{pro13--bis}
\BBp\left(d_H(\bbx_n , {\BBm}) > \epsilon\right)\leq \BBp\left(d_H(\bbx_{kr} , {\BBm}) > \epsilon\right).
\end{equation}
It finally remains to prove that
\begin{equation}\label{pro13}
\BBp\left(d_H(\bbx_{kr} , {\BBm}) > \epsilon\right)\leq \BBp\left(d_H(\bby_{k} , {\BBm}_{dr}) > \epsilon\right).
\end{equation}
For this, let $X_j\in \bbx_{kr} $ and $x\in \BBm$. Then there exist $l$ and $i$ such that $X_j$ is the $l$-th component of the vector $Y_{i,r}$.
We claim also that there exists ${\tilde x}\in \BBm_{dr}$ such that $x$ is the $l$-th component of the vector ${\tilde x}$.
{{In fact, let $\pi_l: \BBr^{dr}\rightarrow \BBr^d$ be the projection onto the $l$-th factor. It follows from an elementary property of the support, by the continuity of $\pi_l$ and the closure of ${\BBm_{dr}}$\footnote{{We are grateful to one of the referees of this paper for simplifying a previous argument.}} that
$$
\BBm={\mbox{supp}} (X_j) = \overline{\pi_l({\mbox{supp}}(Y_{i,r}))}=
\overline{\pi_l(\BBm_{dr}}) = \pi_l(\BBm_{dr}),
$$
where $\overline{A}$ denotes, as before, the closure of the set $A$. So, in particular, any $x\in \BBm$ is $x=\pi_l({\tilde x})$ for some ${\tilde x} \in \BBm_{dr}.$
}}
From this, we deduce that, for any $X_j\in \bbx_{kr} $ and $x\in \BBm$, there exist $1\leq i\leq k$ and ${\tilde x}\in \BBm_{dr}$ such that, a.s.,
$$
\|X_j-x\|\leq \|Y_{i,r}-{\tilde x}\|.
$$
Hence,
$$
\inf_{X_j\in \bbx_{kr} }\|X_j-x\|  \leq  \inf_{Y_{i,r} \in \bby_k}\|Y_{i,r}-{\tilde x}\|\leq d_H(\bby_{k} , {\BBm}_{dr}).
$$
Consequently, since $\BBp(\bbx_{kr} \subset \BBm)=1$,
$$
d_H(\bbx_{kr} , {\BBm})=\sup_{x\in \BBm}\inf_{X_j\in \bbx_{kr} }\|X_j-x\|  \leq d_H(\bby_{k} , {\BBm}_{dr}).
$$
From this we get (\ref{pro13}). Now (\ref{pro13}) together with 
(\ref{pro13--bis}) prove (\ref{pro13bis}). The proof of this proposition is complete.
\end{proof}

%%%%%%%%%%%%%%%%%%%%%%%%%%%%%%%%%%%%%%%%%

\section{Asymptotically $(\epsilon,\alpha)$-dense sequences of random variables}\label{examples}

As indicated in the introduction, our main goal is to find conditions under which a sequence $\bbx$ is asymptotically $(\epsilon,\alpha)$-dense in the common support  (see Definition \ref{densite}).
In this section, we give conditions and several examples of dependent random variables for which this is the case. This property is established every time by means of Proposition \ref{pro1} applied with suitable choices of sub-sequences $k$ and $r$ of $n$, and for all these examples, it holds that for any $\epsilon>0$,
\begin{equation*}
\lim_{n\rightarrow \infty}\BBp\left(d_H(\bby_{k} , {\BBm}_{dr}) > \epsilon\right)=
\lim_{n\rightarrow \infty}\BBp\left(d_H(\bbx_n , {\BBm}) > \epsilon\right)=0.
\end{equation*}
All proofs of the  propositions listed in this section appear in Section \ref{proofs}.

\subsubsection{Stationary $m$-dependent sequence on a compact set}
 
Recall that the sequence $(X_i)_{i\in \BBt}$   is $m$-dependent for some $m\geq 0$ if  the two
$\sigma$-fields $\sigma (X_i,\,\,i\leq k)$ and $\sigma (X_{i}, i\geq k+m+1)$ are independent for every $k$.
 In particular, $0$-dependent is the same as
independent.
\begin{example} ($m$-dependent sequence).\rm \\
Let $(T_i)_{i\in \BBn}$ be a sequence of i.i.d. random variables with values in $\BBr^d$. Let $h$ be a real-valued function  defined on ${\BBr^{dm}}$. The
stationary sequence $(X_n)_{n\in \BBn}$ defined by $X_n=h(T_n,T_{n+1},\cdots,T_{n+m})$
is a stationary sequence of $m$-dependent random variables. 
\end{example}

Define, for $m\in \BBn\setminus\{0\}$, $\epsilon>0$, and for 
$Y_{1,m}=(X_1,\cdots,X_m)^t$, as in the introduction, the concentration coefficient of the vector $Y_{1,m}$,
\begin{equation}\label{rhom}
\rho_{m}(\epsilon)=\inf_{x\in {\BBm}_{dm}}\BBp\left(\|Y_{1,m}-x\|\leq \epsilon\right).
\end{equation}

The following {proposition} gives conditions on $\rho_m(\epsilon)$ under which the asymptotically $(\epsilon,\alpha)$-dense property evoked in Definition (\ref{densite}) is satisfied.

\begin{pro}\label{coro2}
Let $(X_i)_{i\in \BBt}$ be a stationary sequence of $m$-dependent, $\BBr^d$-valued random vectors. Suppose that $X_1$
is with compact support $\BBm$. Let $\epsilon_0>0$ be fixed. Suppose that for any $0<\epsilon<\epsilon_0$, there exists a
strictly positive constant $\kappa_{\epsilon}$ such that,
$$
\rho_{m+1}(\epsilon) \geq \kappa_{\epsilon},
$$
then   it holds for any $0<\epsilon<\epsilon_0$ and any $n\geq m+1$,
\begin{eqnarray*}
&&\BBp\left(d_H(\bbx_n , {\BBm}) > \epsilon\right) 
\leq \frac{(1- \kappa_{\frac{\epsilon}{2}})^{[\frac{1}{2}[\frac{n}{m+1}]]}}{\kappa_{\frac{\epsilon}{4}}},
\end{eqnarray*}
where $[\cdot]$ denotes the integer part.  
Consequently, for any $\alpha\in ]0,1[$ and any $n\geq n_0$, where
$$n_0= \frac{2(m+1)}{\kappa_{\frac{\epsilon}{2}}} \left(\log\left(\frac{1}{\alpha }\right)+ \log\left(\frac{1}{ \kappa_{\frac{\epsilon}{4}} } \right)\right)+ 3(m+1),$$ 
$d_H(\bbx_n , {\BBm}) \leq \epsilon$ with probability at least $1-\alpha$.
\end{pro}
The requirements of Proposition \ref{coro2} prove  that the sequence $(X_n)_{n\in \BBt}$ is asymptotically $(\epsilon,\alpha)$-dense in $\BBm$ with threshold
$n_0$ as above.

%%%%%%%%%%%%%%%%%%%%%%%%%%%%%%%%%%%%%%%%%%%%%%%%

\subsubsection{Stationary $m$-Approximable random 
variables on a compact set}\label{sectApp} 

{{
In this section we discuss, in the spirit of \cite{hk},  some examples of stationary compactly supported random variables $(X_n)_{n\in \bbz}$  
that can be approximated by $m$-dependent stationary sequences.  More precisely, the article \cite{hk}  introduced the notion of $L^p$-$m$-approximable sequence. This notion is related to $m$-dependence (see  Definition 2.1 of \cite{hk}) and is different from 
mixing (see Paragraph \ref{secbeta} below for a definition of mixing).  The idea is to construct, for $m\in \BBn$, a stationary sequence $(X_n^{(m)})_{n\in \bbz}$, $m$-dependent, compactly supported, for which the Hausdorff distance between the two sets ${\bbx}_n$ and ${{\bbx}_n^{(m)}}:=\{X_1^{(m)},\cdots, X_n^{(m)}\}$ is suitably controlled. 
In our case, it is not necessary that $X_1$ and $X_1^{(m)}$ have the same distribution, but what we need is that both  $X_1$ and $X_1^{m}$ have the same compact support. This will give us more choices for the construction of the sequence $(X_n^{(m)})_{n\in \bbz}$ which can be obtained by the method of coupling or by a  truncation argument (see \cite{hk} for more details). For our purpose, we shall use a truncation.  

More precisely, we will consider the sequence:
\begin{equation}\label{approI}
X_n=f(\epsilon_n,\epsilon_{n-1},\cdots)
\end{equation}
where $(\epsilon_i)_{i\in \bbz}$ is an i.i.d sequence with values in some measurable space $S$ and $f$ is a real bounded  function defined on $S^{\infty}$. The sequence $(X_n^{(m)})_{n\in \bbz}$ constructed, from  $(X_n)_{n\in \bbz}$, by truncation is:
\begin{equation}\label{approII}
X_n^{(m)}=f(\epsilon_n,\cdots,\epsilon_{n-m},0,\cdots).
\end{equation}
Clearly $(X_n^{(m)})_{n\in \bbz}$ is a stationary, $m$-dependent sequence and with the same compact support as  $(X_n)_{n\in \bbz}$ as soon as $f$ is bounded. We will assume thus that $f$ is bounded; that is $\|f\|_{\infty}=\sup_{x\in S^{\infty}}|f(x)|<\infty.$ As before, $\BBm$ will be the common support of those two sequences.

Now we need an additional assumption on $f$ in order to ensure a good control of the Hausdorff distance between the two sets ${\bbx}_n$ and ${{\bbx}_n^{(m)}}$. 
We suppose that $f$ is a real-valued bounded function and it satisfies the following ``Lipschitz type'' assumption (stated in \cite{hk}): there exists a decreasing sequence $(c_m)_{m\in \BBn}$ tending to $0$ as $m$ tends to infinity, such that
\begin{equation}\label{approIII}
|f(a_{m+1},\cdots,a_1,x_0,\cdots)-f(a_{m+1},\cdots,a_1,y_0,\cdots)| \leq c_m |f(x_0,x_{-1},\cdots)-f(y_0,y_{-1},\cdots)|,
\end{equation}
for any numbers $a_l,x_i,y_i\in S$, $l\in\{1,m+1\}$ and $ i\leq 0$. This assumption is satisfied, for instance, by some autoregressive models of order 

The next lemma proves that the truncated sequence $(X_n^{(m)})_{n\in \bbz}$ is a Hausdorff approximation of the original sequence $(X_n)_{n\in \bbz}$. 

\begin{lem}\label{lemmapro} 
Let $\epsilon>0$ be fixed, $(\epsilon_i)_{i\in \bbz}$ be an i.i.d. sequence, $f$ a bounded function satisfying (\ref{approIII}), and let $(X_n)_{n\in \bbz}$ and $(X_n^{(m)})_{n\in \bbz}$ 
 be the associated sequences as in (\ref{approI}) and (\ref{approII}) respectively.
Let $m \in \BBn$ be such that
$$
2c_m \|f\|_{\infty} <  \epsilon,
$$
where $\|f\|_{\infty}$ is the supremum of $f$.
Then 
$
d_H({\bbx}^{(m)}_n, {\bbx}_n)<\epsilon,\,\,\, a.s.
$
\end{lem}

In view of Lemma \ref{lemmapro}, the condition $\lim_{m\rightarrow \infty}c_m=0$  is enough  to approximate,  in the Hausdorff sense, the sequence  $(X_n)_{n\in \bbz}$ by the truncated sequence $(X_n^{(m)})_{n\in \bbz}$. 
\begin{proof} We recall that,
$$
d_H({\bbx}^{(m)}_n, {\bbx}_n)=\max\left(\max_{1\leq i\leq n}\min_{1\leq j\leq n}|X_i-X_j^{(m)}|, \max_{1\leq i\leq n}\min_{1\leq j\leq n}|X_j-X_i^{(m)}| \right).
$$
Hence,
\begin{eqnarray*}
 && \BBp(d_H({\bbx}^{(m)}_n, {\bbx}_n)\geq \epsilon) \\
 && \leq \BBp(\max_{1\leq i\leq n}\min_{1\leq j\leq n}|X_i-X_j^{(m)}|\geq \epsilon) + \BBp( \max_{1\leq i\leq n}\min_{1\leq j\leq n}|X_j-X_i^{(m)}|\geq \epsilon) \\
 &&  \leq 2 n \max_{1\leq i \leq n}\BBp( |X_i-X_i^{(m)}|\geq \epsilon).
\end{eqnarray*}
Now,
$$
|X_i-X_i^{(m)}| \leq c_m|f(\epsilon_{n-m-1},\epsilon_{n-m-2},\cdots)-f(0,0,\cdots)| \leq 2c_m \|f\|_{\infty}.
$$
Consequently the event 
$( |X_i-X_i^{(m)}|\geq \epsilon)$ implies that 
$\epsilon \leq 2c_m \|f\|_{\infty}$ and then the probability
$
\BBp( |X_i-X_i^{(m)}|\geq \epsilon) 
$ vanishes
whenever $m$ satisfies $2c_m \|f\|_{\infty}<\epsilon$. We conclude that, for such $m$, $\BBp(d_H({\bbx}^{(m)}_n, {\bbx}_n)\geq \epsilon)=0$. \end{proof}

We can now apply Proposition \ref{coro2} to the $m$-dependent sequence $(X_n^{(m)})_{n\in \bbz}$, combined with Lemma \ref{lemmapro}, to establish asymptotic $(\epsilon,\alpha)$-density for $(X_n)_{n\in\bbz}$. Doing this, we obtain the following result under a suitable control of the following concentration coefficient, related to the truncated sequence $(X_n^{(m)})_{n\in \bbz}$ (as defined in (\ref{approII})),
\begin{equation}\label{rhoconcen}
 \rho_{m+1}^{(m)}(\epsilon)= \inf_{x \in \BBr^{m+1}}\BBp(\|Y_{1,m+1}^{(m)}-x\|\leq \epsilon),
\end{equation}
with $Y_{1,m+1}^{(m)}=(X_1^{(m)},\cdots,X_{m+1}^{(m)})^t$.
\begin{pro}\label{proapp}
     Let $(X_n)_{n\in \bbz}$ and $f$ be as in the statement of Lemma \ref{lemmapro}. Let $\epsilon_0>0, \epsilon \in ]0,\epsilon_0[$ be fixed and $m \in \BBn$ be such that $2c_m \|f\|_{\infty} <  \epsilon$. Suppose that the concentration coefficient $\rho_{m+1}^{(m)}(\epsilon)$ related to the truncated sequence $(X_n^{(m)})_{n\in \bbz}$, and defined in  (\ref{rhoconcen}), satisfies
     $$
     \rho_{m+1}^{(m)}(\epsilon)\geq {\kappa_{\epsilon}},
     $$
     for some ${\kappa_{\epsilon}}>0$. Then for any $n\geq m+1$,
     $$
\BBp(d_H(\bbx_n, \BBm)\geq 2\epsilon)\leq  \frac{(1- \kappa_{\frac{\epsilon}{2}})^{[\frac{1}{2}[\frac{n}{m+1}]]}}{\kappa_{\frac{\epsilon}{4}}}
     $$
and a similar conclusion to Proposition \ref{coro2} is true for such $m$.     
\end{pro}

%\noindent{\bf{Proof of Proposition \ref{proapp}.}}
\begin{proof}
This follows from Lemma \ref{lemmapro}
together with Proposition \ref{coro2} and the fact that,
\begin{eqnarray*}
&& \BBp(d_H(\bbx_n, \BBm)\geq 2\epsilon)\leq \BBp(d_H({\bbx}_n, \bbx^{(m)}_n)+ d_H({\bbx}^{(m)}_n, \BBm) \geq 2\epsilon) \\
&&\leq \BBp(d_H({\bbx}_n, {\bbx}^{(m)}_n)\geq\epsilon)+ \BBp(d_H({\bbx}^{(m)}_n, \BBm) \geq\epsilon).
\end{eqnarray*}
%$\hfill\Box$
\end{proof}
}}

\subsubsection{Stationary $\beta$-mixing sequence on a compact set}\label{secbeta}
Recall that the stationary sequence $(X_n)_{n\in \BBn}$ is $\beta$-mixing if $\beta_n$ tends to $0$ when $n$ tends to infinity
where the coefficients $(\beta_n)_{n>0}$ are defined by, (see \cite{brc} and \cite{Yu} for the following expression of $\beta_n$),
\begin{equation}\label{betabelow}
\beta_n=\sup_{l\geq 1}\BBe  \left\{\sup\left| \BBp\left(B|\sigma(X_1,\cdots,X_l)\right) -\BBp(B)\right|,\,\, B\in \sigma(X_i,\,\, i\geq l+n)\right\}.
\end{equation}
The following corollary gives conditions on the behavior of the two sequences $(\rho_n(\epsilon))_{n>0}$ and $(\beta_n)_{n>0}$ under which   the asymptotically $(\epsilon,\alpha)$-dense property of Definition (\ref{densite}) is satisfied.

\begin{pro}\label{beta}
Let $(X_n)_{n\geq 0}$ be a  stationary $\beta$-mixing sequence. Suppose that $X_1$ is supported on a compact set $\BBm$. Then it holds, for any $\epsilon>0$ and any sequences
$k_n$ and $r_n$ such that $k_nr_n\leq n$, 
$$
\BBp\left(d_H(\bbx_n , {\BBm}) > \epsilon\right) \leq 
\frac{k_n^2\beta_{r_n}+ k_n\exp\left(-[\frac{k_n}{2}]\rho_{r_n}(\epsilon/2)\right)}{k_n\rho_{r_n}(\epsilon/4)}.
$$
Suppose moreover that for some $\beta>1$, and any $\epsilon>0$ small enough,
$$\lim_{m\rightarrow \infty} \rho_m(\epsilon)\frac{e^{m^{\beta}}}{m^{1+\beta}}=\infty,\,\,\,
{\mbox{and}}\,\,\,
\lim_{m\rightarrow\infty} \frac{e^{2m^{\beta}}}{m^2}\beta_m=0.
$$
Then $(X_n)_{n\geq 0}$ is asymptotically $(\epsilon,\alpha)$-dense in $\BBm$.
\end{pro}

{{ In the proof of Proposition \ref{beta} given in Section \ref{proofs}, we construct two sequences $(k_n)_n$ and $(r_n)_n$ for which 
$$
\lim_{n\rightarrow \infty} \frac{k_n^2\beta_{r_n}+ k_n\exp\left(-[\frac{k_n}{2}]\rho_{r_n}(\epsilon/2)\right)}{k_n\rho_{r_n}(\epsilon/4)} =0. 
$$
The threshold $n_0$ is not explicitly calculated but it is that integer for which 
$$
 \frac{k_n^2\beta_{r_n}+ k_n\exp\left(-[\frac{k_n}{2}]\rho_{r_n}(\epsilon/2)\right)}{k_n\rho_{r_n}(\epsilon/4)} \leq \alpha, 
$$ 
for all $n\geq n_0$. 
}}

\subsubsection{Stationary weakly dependent sequence on a compact set}
We suppose here that $(X_{{i}})_{i \in \BBt}$ is a stationary sequence such that $X_1$ takes values in a compact support $\BBm$.
We suppose also that this sequence is weakly dependent in the sense of \cite{DLLY}. {{ More precisely, we suppose that it satisfies the following definition.
\begin{defi}\label{psidef}
We say that the sequence $(X_n)_{n\in \BBt}$ is $(\BBl_{\infty}, \Psi)$-weakly dependent
if there exists a non-increasing function $\Psi$ such that
$\lim_{r\rightarrow\infty}\Psi(r)=0$, that for any measurable functions $f$ and $g$ bounded (respectively by $\|f\|_{\infty}$ and $\|g\|_{\infty}$)
and for any $ i_1\leq \cdots \leq i_k<i_{k}+r\leq i_{k+1}\leq \cdots\leq i_n$ one has
\begin{eqnarray}\label{psidep}
\left|\Cov\left(\frac{f(X_{i_1},\cdots,X_{i_k})}{\|f\|_{\infty}}, \frac{g(X_{i_{k+1}},\cdots, X_{i_n})}{\|g\|_{\infty}}\right) \right| \leq \Psi(r).
\end{eqnarray}
\end{defi}
\noindent See  Definition 2.2 in \cite{DDL} for a more general setting. }}

{{The dependence condition  in Definition \ref{psidef} is weaker than the Rosenblatt strong mixing dependence \cite{ros1}. Let us briefly explain this. The $\alpha$-mixing coefficient between the two sigma-fields ${\mathcal A}$ and ${\mathcal B}$ is  defined as
$$
\alpha(\mathcal A, \mathcal B)=\sup_{A\in {\mathcal A} ,B\in \mathcal B}|\BBp(A\cap B)-\BBp(A)\BBp(B)|.
$$
The sequence $(X_n)_{n\in \BBt}$
is strongly mixing if its coefficient $\alpha_n$ defined, for $n\geq 1$, by
$$
\alpha_n=\sup_{k \in \BBt} \alpha({\mathcal P}_k, {\mathcal F}_{k+n}) 
$$
tends to $0$ as $n$ tends to infinity, with ${\mathcal P}_k=\sigma(X_i,\, i\leq k)$  and 
${\mathcal F}_{k+n}=\sigma(X_i,\, i\geq  k +n)$. An equivalent formula for $\alpha_n$, using the covariance between some functions, is stated in Theorem 4.4 of \cite{BR2007} 
$$
\alpha_n=\frac{1}{4}\sup\left\{\frac{\Cov(f,g)}{\|f\|_{\infty}\|g\|_{\infty}},\,\, f\in L_{\infty}({\mathcal P}_k), g\in L_{\infty}({\mathcal F}_{k+n})\right\},
$$
where $L_{\infty}({\mathcal A})$ denotes the set of bounded functions ${\mathcal A}$-measurables for some $\sigma$-fields $\mathcal A$.  
It follows from this formula that strongly mixing sequences are $(\BBl_{\infty}, \Psi)$-weakly dependent, as stated in Definition  \ref{psidef} (with $\Psi(r)=\alpha_r$ for  $r>0$). The converse is however not necessarily true (see \cite{DDL}). }}

{{We can now state our result for stationary weakly dependent sequences.}}

\begin{pro}\label{psi}
Let $(X_n)_{n\in \BBt}$ be a sequence of stationary, $(\BBl_{\infty}, \Psi)$-weakly dependent in the sense of Definition (\ref{psidef}). Suppose that $X_1$ is supported on a compact set $\BBm$. 
 Then it holds, for any $\epsilon>0$ and any sequences
$k_n$ and $r_n$ such that $k_nr_n\leq n$, 
$$
\BBp\left(d_H(\bbx_n , {\BBm}) > \epsilon\right) \leq 
\frac{k_n^2\Psi({r_n})+ k_n\exp\left(-[\frac{k_n}{2}]\rho_{r_n}(\epsilon/2)\right)}{k_n\rho_{r_n}(\epsilon/4)}.
$$
Suppose moreover that, for some $\beta>1$, and any positive $\epsilon$ small enough,
$$\lim_{m\rightarrow \infty} \rho_m(\epsilon)\frac{e^{m^{\beta}}}{m^{1+\beta}}=\infty,\,\,\,
{\mbox{and that}}\,\,\,
\lim_{m\rightarrow\infty} \frac{e^{2m^{\beta}}}{m^2}\Psi(m)=0.
$$
Then this sequence $(X_n)_{n\in \BBt}$ is asymptotically $(\epsilon,\alpha)$-dense in $\BBm$.
\end{pro}
Here again the threshold $n_0$ is not explicitly calculated but it can be given by an inequality, as it is the case for $\beta$-mixing.

\subsection{Comparison with the i.i.d case} \label{rem}
We can compare the bounds of Proposition \ref{coro2} (respectively Propositions \ref{beta} and \ref{psi}) with what is already obtained in the i.i.d case (see \cite{Ch}, \cite{CR2004}, \cite{CR2009}). That is, we restrict to the case when  $m=0$ (respectively, $k_n=n$ and $r_n=1$, $\beta_n=\Psi(n)=0$ for $n\geq 1$). Suppose we are in this situation, and that moreover $\rho_1(\epsilon)$ has a strictly positive lower bound, say $\kappa_{\epsilon}$. Then all the conclusions of the three propositions above give the same upper bound for 
$\BBp\left(d_H(\bbx_n , {\BBm}) > \epsilon\right)$ which is,
\begin{equation}\label{upperiid}
\frac{\exp(-[\frac{n}{2}]\kappa_{\frac{\epsilon}{2}})}
{\kappa_{\frac{\epsilon}{4}}}.
\end{equation}
Now we suppose, as already done in the i.i.d case, that the 
 $(a,b)$-standard assumption\footnote{{The 
 $(a,b)$-standard assumption was used, in the i.i.d context,  for set estimation problems under Hausdorff distance (\cite{CR2004}, \cite{CR2009}) and also for a statistical analysis of persistence diagrams (\cite{Ch}, \cite{Fasy}). }} is satisfied, i.e,
$
\kappa_{\epsilon}= a \epsilon^{b},
$
for some $a>0, b>0$ and for positive $\epsilon$ small enough. 
Then clearly, an upper bound for (\ref{upperiid}) is 
$$
C \frac{\exp(-c n \epsilon^b)}{\epsilon^b},
$$
for some positive constants $C$ and $c$ (independent of $n$ and of $\epsilon$) as was already found in the i.i.d case ({{see for instance the upper bound (3.2) in \cite{CR2009}}}). 
Finally, we have to check that the requirements of Propositions \ref{coro2},  \ref{beta} and \ref{psi} are satisfied under the 
 $(a,b)$-standard assumption
(i.e. when $\rho_1(\epsilon)\geq a \epsilon^{b}$). Since we are in the case when $m=0$ in Proposition \ref{coro2} and when $\beta_n=0, \Psi_n=0, \, n\geq 1$ in Propositions 
\ref{beta}, \ref{psi}, we have only 
to check that  the 
 $(a,b)$-standard assumption ensures, for i.i.d. random variables,
 $\lim_{m\rightarrow \infty} \rho_m(\epsilon)\frac{e^{m^{\beta}}}{m^{1+\beta}}=\infty$. We deduce from the inequality
 $$
 \|(X_1,\cdots,X_m)^t -(x_1,\cdots,x_m)^t\|^2 \leq 
 m \max_{1\leq i\leq m}\|X_i-x_i\|^2,
 $$
 that
 $$
 \BBp\left(m \max_{1\leq i\leq m}\|X_i-x_i\|^2 \leq \epsilon^2\right)\leq  \BBp\left(\|(X_1,\cdots,X_m)^t -(x_1,\cdots,x_m)^t\|^2\leq \epsilon^2\right).
 $$
 Now,
 $$
 \BBp\left(m \max_{1\leq i\leq m}\|X_i-x_i\|^2 \leq \epsilon^2\right) = \left(\BBp\left(\|X_1-x_1\| \leq \epsilon/{\sqrt{m}}\right)\right)^m.
 $$
 Hence,
 $\rho_m(\epsilon)\geq \rho_1^m(\epsilon/{\sqrt{m}}).$ 
We get, combining this bound with the $(a,b)$-standard assumption, 
\begin{eqnarray*}
&& a^m \frac{\epsilon^{bm}}{m^{bm/2}}\frac{e^{m^{\beta}}}{m^{1+\beta}}
\leq \rho_m(\epsilon)\frac{e^{m^{\beta}}}{m^{1+\beta}}.
\end{eqnarray*}
The left term tends to infinity as $n$
goes to infinity (since $\beta>1$), hence 
$\displaystyle\lim_{m\rightarrow \infty} \rho_m(\epsilon)\frac{e^{m^{\beta}}}{m^{1+\beta}}
=\infty$. 

As a main conclusion, the previous three propositions
generalise well the  i.i.d case, even without assuming the
$(a,b)$-standard assumption.

%%%%%%%%%%%%%%%%%%%%%%%%%%%%%%%%%%%%%%%%%%%%%%%%%%%%%

\section{Application to stationary Markov chains on a compact state space}\label{sectionMarkov}

This section gives conditions on stationary Markov chains on a compact state space so that they are asymptotically $(\epsilon,\alpha)$-dense in this state space. Those conditions can be checked by studying the $\beta$-mixing
properties of these Markov chains and by applying 
Proposition \ref{beta} above. We choose however in this section to be even more precise by adopting specific models and carrying out explicit calculations.

\vskip 5pt
Let $(X_n)_{n\geq 0}$ be an homogeneous   Markov chain satisfying the following two assumptions.
\begin{itemize}
\item[$({\mathcal A}_1)$] This Markov chain has an invariant measure $\mu$ with compact support $\BBm$
(and then the chain is stationary).
\item[$({\mathcal A}_2)$] The transition probability kernel $K$, defined for $x\in \BBm$, by
 $$
 K(x,\cdot) = \BBp(X_1\in \cdot |X_0=x)%=\BBp(F(x,Z_1)\in \cdot )
 $$
 is absolutely continuous with respect to some measure $\nu$ on $\BBm$, i.e. there exists a positive measure $\nu$ and a positive function $k$ such that for any $x\in \BBm$, $
 K(x,dy)=k(x,y)\nu(dy)$. Moreover, for  some $b>0$ and $\epsilon_0>0$,
 $$
 V_{d}:=\inf_{x\in \BBm}\inf_{0<\epsilon<\epsilon_0}\left(\frac{1}{\epsilon^b}\int_{B(x,\epsilon)\cap \BBm }\nu(dx_1)\right)>0
 $$
 and  that there exists a positive constant $\kappa$ such that
 $\displaystyle
 \inf_{x\in \BBm,\,y\in \BBm}k(x,y)\geq \kappa>0.
 $.
\end{itemize}

{{
\begin{pro}\label{coro1}
Suppose that Assumptions $({\mathcal A}_1)$ and $({\mathcal A}_2)$ are satisfied for some Markov chain $(X_n)_{n\geq 0}$. Then, for any $n\geq 1$ and any positive $\epsilon$ small enough,
$$ \BBp_{\mu}\left(d_H(\bbx_n ,\BBm)>\epsilon \right) \leq \frac{4^b(1-\kappa \epsilon^bV_{d}/2^b)^{n}}{\kappa\epsilon^b V_{d}}.$$ 
Consequently  this Markov chain  is asymptotically ($\epsilon,\alpha$)-dense in $\BBm$ with a threshold $n_0$ given by$$
n_0=\frac{2^b}{\kappa \epsilon^b V_d}\left(\ln\left(\frac{4^b}{\kappa \epsilon^b V_d}\right)+\ln\left(\frac{1}{\alpha}\right)  \right),
$$
and $V_d$ is as introduced in Assumption $({\mathcal A}_2)$. 
\end{pro}
}}

The proof, and some key lemmas, are deferred to Section \ref{proof}. 

We next give examples of Markov chains satisfying the requirements of Proposition \ref{coro1}. Those examples concern stationary Markov chains on  the balls and stationary Markov chains on the circles.
\subsection{Stationary Markov chains on a ball of $\BBr^d$}

\subsubsection{Random difference equations}
Let $(X_n)_{n\geq 0}$ be a Markov chain defined, for $n\geq 0$,  by
\begin{equation}\label{SRE}
X_{n+1}=A_{n+1}X_{n}+ B_{n+1},
\end{equation}
where  $A_{n+1}$ is a $(d\times d)$-matrix, $X_n\in \BBr^d,\,\, B_n\in \BBr^d$, $(A_n,B_n)_{n\geq 1}$ is an i.i.d. sequence
independent of $X_0$. Recall that for a matrix $M$, $\|M\|$ is the operator norm defined by $\|M\|=\sup_{x\in \BBr^d,\,\|x\|=1}\|Mx\|$.
It is well known  that
for any $n\geq 1$, $X_n$ is distributed as $\sum_{k=1}^n A_1\cdots A_{k-1}B_k + A_1\cdots A_{n}X_0$, see for instance \cite{kesten}. It is also well-known that
the following conditions (see \cite{goldie, kesten1})
\begin{equation}\label{csre}
\BBe(\ln^+\|A_1\|)<\infty,\,\, \BBe(\ln^+\|B_1\|)<\infty,\,\, \lim_{n\rightarrow \infty}\frac{1}{n}\ln\|A_1\cdots A_n\|<0\,\,\, a.s.,
\end{equation}
ensure the existence of a stationary solution to (\ref{SRE}), and that $\|A_1\cdots A_n\|$ approaches $0$ exponentially fast. If in addition $\BBe\|B_1\|^{\beta}<\infty$ for some $\beta>0$, then the series $R:=\sum_{i=1}^{\infty}A_1\cdots A_{i-1}B_i$ converges a.s.
and the distribution of $X_n$ converges to that of $R$, independently of $X_0$. The distribution of $R$ is then that of the stationary measure of the chain.\\
\\
{\it Compact state space.}
  If $\|B_1\|\leq c<\infty$ for some fixed $c$, then this stationary Markov chain is $\BBm$-compactly supported. In particular if
    $\|A_1\|\leq \rho <1$ for some fixed $\rho$, then $\BBm$ is included in the ball  $B_d(0, \frac{c}{1-\rho})$ of $\BBr^d$.\\
\\
{\it Transition kernel.}
Suppose that, for any $x \in \BBm$, the random vector $A_1x+B_1$ has a density $f_{A_1x+B_1}$ with respect to the Lebesgue measure (here $\nu$ is  the Lebesgue measure) satisfying $\inf_{x,\,\,y\in \BBm}f_{A_1x+B_1}(y)\geq \kappa$, then
$
k(x,y)=  f_{A_1x+B_1}(y)\geq \kappa>0.
$\\
\\
We collect all the above results in the following corollary.

\begin{coro}\label{corosre} Suppose that in the model (\ref{SRE}), conditions (\ref{csre}) are satisfied, and moreover $\|B_1\|\leq c<\infty$.
If the density of $A_1x+B_1$; $f_{A_1x+B_1}$, satisfies $\inf_{x,\,\,y\in \BBm}f_{A_1x+B_1}(y)\geq \kappa>0$ for some positive $\kappa$,
then assumptions (${\mathcal A}_1$) and (${\mathcal A}_2$) are  satisfied with $b=d$ and $\nu$ is the Lebesgue measure on $\BBr^d$. 
\end{coro}

\subsubsection*{Example. The AR(1) process in $\BBr$.}
We consider a particular case of the Markov chain as defined in (\ref{SRE}) with $d=1$, where, for each $n$, $A_n=\rho$ with $|\rho|<1$. We obtain
the standard first order linear Auto-Regressive process, that is
$$
X_{n+1}=\rho X_n + B_{n+1},
$$
we suppose that
\begin{itemize}
    \item  $B_1$ has a density function $f_B$ supported on $[-c,c]$ for some $c>0$ with $\kappa:=\inf_{x\in [-c,c]}f_B(x)>0$
    \item $X_0 \in [\frac{-c}{1-|\rho|}, \frac{c}{1-|\rho|}]$
\end{itemize}
This Markov chain evolves in a compact state space which is a subset of $ [\frac{-c}{1-|\rho|}, \frac{c}{1-|\rho|}]$. Thanks to Corollary \ref{corosre}, $(X_n)_n$ admits a stationary measure $\mu$. We have, moreover,
$$
k(x,y)=f_{B_1}(y-\rho x) \geq \kappa,\,\, \forall\,\, x\in \BBm,\,\, \forall\,y\in \BBm.
$$
Assumptions (${\mathcal A}_1$) and (${\mathcal A}_2$) are then satisfied with $b=1$ and $\nu$ is the Lebesgue measure on $\BBr$.

\subsubsection*{Example. The AR($k$) process in $\BBr$.}
The AR($k$) is defined by,
$$
Y_n= \alpha_1 Y_{n-1} + \alpha_2 Y_{n-2}+\cdots+ \alpha_k Y_{n-k} + \epsilon_n,
$$
where $\alpha_1,\cdots,\alpha_k \in \BBr$. Since this model can be written in the form of (\ref{SRE}) with $d=1$,
$$
X_n= (Y_n,Y_{n-1},\cdots,Y_{n-k+1})^t,\,\,\,\, B_n= (\epsilon_n,0,\cdots,0)^t,\,\,\, A_n=\left(
                                                                                           \begin{array}{ccc}
                                                                                             \alpha_1 & \cdots & \alpha_k \\
                                                                                             I_{k-1} &  & 0 \\
                                                                                           \end{array}
                                                                                         \right)
$$
all the above results, for random difference equations, apply under the corresponding assumptions. In particular the process AR($2$) is stationary as soon as $|\alpha_2|<1$ and $\alpha_2+ |\alpha_1|<1$.

\subsection{The M\"obius Markov chain on the circle}\label{sectionmobius}
Our aim is to give an example of a Markov chain on the unit circle, known as M\"obius Markov chain, satisfying the requirements of Proposition  \ref{coro1}. This Markov chain $(X_n)_{n\in \BBn}$ is introduced in
\cite{Kato} and is defined as follows:
\begin{itemize}
\item $X_0$ is a random variable which takes values on the unit circle.
\item For $n\geq 1$,
$$
X_n=\frac{X_{n-1} + \beta}{\beta X_{n-1}+1}\epsilon_n,
$$
where $\beta\in ]-1,1[$ and $(\epsilon_n)_{n\geq 1}$ is a sequence of i.i.d. random variables which are independent
of $X_0$ and distributed as the wrapped Cauchy distribution with a common density with respect to the arc length measure $\nu$ on the unit circle $\partial B(0,1)$,
$$
f_{\varphi}(z)= \frac{1}{2\pi}\frac{1-\varphi^2}{|z-\varphi|^2},\,\, \varphi\in [0,1[ \,\, {\mbox{fixed}},\,\,\, z\in \partial B(0,1).
$$
\end{itemize}
The following proposition holds.

\begin{pro}\label{mobius}
Let $(X_n)_{n\geq 0}$ be the M\"obius Markov chain on the unit circle as defined above.
Then this Markov chain admits a unique invariant distribution, denoted by $\mu$.
If $X_0$ is distributed as $\mu$ then the set
$\bbx_n =\{X_1,\cdots,X_n\}$ converges in probability, as $n$ tends to infinity, in the Hausdorff distance to the unit circle $\partial B(0,1)$,
more precisely, for any $\alpha\in ]0,1[$, any positive $\epsilon$ sufficiently small and any $n\geq \frac{2}{\kappa v \epsilon}\left(\ln(\frac{1}{\alpha})+\ln(\frac{4}{\epsilon\kappa v}) \right) $
$$
d_H(\bbx_n , \partial B(0,1)) \leq \epsilon,
$$
with probability at least $1-\alpha$. Here $v$ is a finite positive constant and
$
\kappa= \frac{1}{2\pi}\frac{1-\varphi}{1+\varphi}.$

\end{pro}

The M\"obius Markov chain of Proposition \ref{mobius} is then asymptotically $(\epsilon,\alpha)$-dense in the unit circle with a threshold $n_0$ given by, 
$$
n_0=  \frac{2}{\kappa v \epsilon}\left(\ln(\frac{1}{\alpha})+\ln(\frac{4}{\epsilon\kappa v}) \right),
$$ $\kappa$ being as in the statement of the proposition while the positive constant $v$ is defined by Formula (\ref{A2-3}) below.

\begin{proof} We have to prove that all the requirements of Proposition \ref{coro1}
are satisfied. Our main reference for this proof is \cite{Kato}.
It is proved there that this Markov chain has a unique invariant measure $\mu$ on the unit circle. {{This measure $\mu$ has full support on $\partial B(0,1)$ (so that Assumption
$({\mathcal A}_1)$ is satisfied with $\BBm=\partial B(0,1)$).}} The task now is to check  Assumption
$({\mathcal A}_2)$.  We have also, for $x\in \partial B(0,1)$,
\begin{eqnarray}\label{A2-1}
&& K(x,dz)= \BBp(X_1\in dz)|X_0=x)=  k(x,z) \nu(dz),
\end{eqnarray}
where $\nu$ is the arc length measure on the unit circle and for $x,z\in \partial B(0,1)$,
$$
k(x,z)= \frac{1}{2\pi}\frac{1-|\phi_1(x)|^2}{|z-\phi_1(x)|^2},
$$
with
$$
\phi_1(x)= \frac{\varphi x + \beta \varphi}{\beta x + 1}.
$$
We obtain, since $\frac{x+\beta}{\beta x+ 1}\in \partial B(0,1)$ whenever $x\in \partial B(0,1)$,
$\displaystyle 
|\phi_1(x)|^2= \varphi^2.
$
Now, for $x,z\in \partial B(0,1)$,
$$
|z-\phi_1(x)|\leq |z|+ |\phi_1(x)|\leq 1+ \varphi.
$$
Hence,
\begin{equation}\label{A2-2}
k(x,z)\geq \frac{1}{2\pi} \frac{1-\varphi^2 }{(1+ \varphi)^2}= \frac{1}{2\pi}\frac{1-\varphi}{1+\varphi}>0.
\end{equation}.
We have now, to check that, for some $\epsilon_0>0$ 
\begin{equation}\label{A2-3}
v:=\inf_{u\in \partial B(0,1)} \inf_{0<\epsilon<\epsilon_0} \left(\epsilon^{-1}\int_{\partial B(0,1) \cap B(u,\epsilon)}\nu(dx_1)\right)>0.
\end{equation}
For this let $u\in \partial B(0,1)$, define
 $\displaystyle\overset{\frown}{AB} = \int_{\partial B(0,1) \cap B(u,\epsilon)}\nu(dx_1).$
 
\begin{figure}[htb]
\begin{center}
\epsfig{file=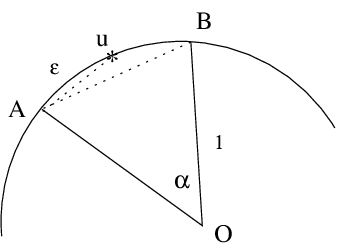,height=1in,width=1.3in,angle=0.0}
\end{center}
\caption{The ball $B(u,\epsilon)$ intersects the unit circle at two points $A$ and $B$. }
\end{figure}

We have (see Figure 3.) $\|u-A\|=\|u-B\|=\epsilon$.
Let $\alpha = \widehat{AOB}$, then
on the one hand $\overset{\frown}{AB}=\alpha$.
On the other hand, since the triangle
$OAu$ is isosceles, with an angle of $\alpha/2$ in $O$, then
$\epsilon  = 2\sin (\alpha/4)$.
We thus obtain
$$
\lim_{\epsilon\rightarrow 0}{1\over\epsilon}{\overset{\frown}{AB}}
= \lim_{\epsilon\rightarrow 0}{\alpha\over\epsilon}
= \lim_{\alpha\rightarrow 0}{\alpha\over 2\sin (\alpha/4)}=2,$$
from this (\ref{A2-3}) is satisfied.

Assumption $({\mathcal A}_2)$  is satisfied thanks to (\ref{A2-1}), (\ref{A2-2}) and (\ref{A2-3}). The proof of Proposition \ref{mobius} is complete by using Proposition \ref{coro1}. 
\end{proof}

\section{Simulations}\label{simu}

The purpose of this section is to simulate a M\"obius Markov process on the unit circle (as defined in Section \ref{sectionmobius})
and to illustrate the results of Proposition \ref{mobius} and of Theorem \ref{mainrecon}. More precisely, we simulate:
\begin{itemize}
\item a random variable,  $X_0$, distributed as the uniform law  on the unit circle $\partial B(0,1)$, that is $X_0$ has the density,
$$
f(z) =\frac{1}{2\pi},\,\,\,\forall\,\, z\in \partial B(0,1).
$$
\item for $n\geq 1$,
$\displaystyle 
X_n=X_{n-1} \epsilon_n,
$, where  $(\epsilon_n)_{n\geq 1}$ is a sequence of i.i.d. random variables which are independent
of $X_0$ and distributed as the wrapped Cauchy distribution with a common density with respect to the arc length measure $\nu$ on the unit circle $\partial B(0,1)$,
$$
f_{\varphi}(z)= \frac{1}{2\pi}\frac{1-\varphi^2}{|z-\varphi|^2},\,\, \varphi\in [0,1[,\,\,\, z\in \partial B(0,1).
$$
\end{itemize}
In this case, it is proved in \cite{Kato}   that this Markov chain is stationary and its  stationary measure is the uniform law  on the unit circle.

\begin{figure}[htb]
\includegraphics[width=13cm]{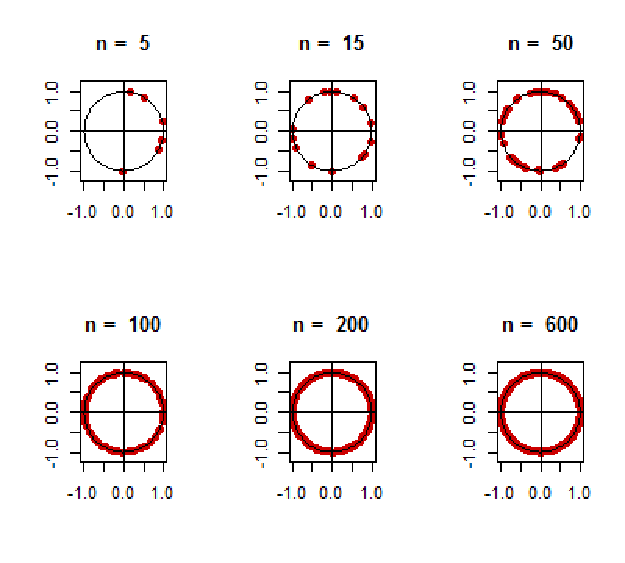}
\caption{Illustrations of the set $\{x_1,\cdots,x_{n}\}$ which is a realisation of the stationary random variables $\bbx_n = \{X_1,\ldots, X_n\}$ for different values of $n$ and with $\varphi=0$.}\label{fig1}
\end{figure}
%\end{center}

\begin{figure}[htb]
\begin{tabular}{ccc}
\includegraphics[width=4.3cm]{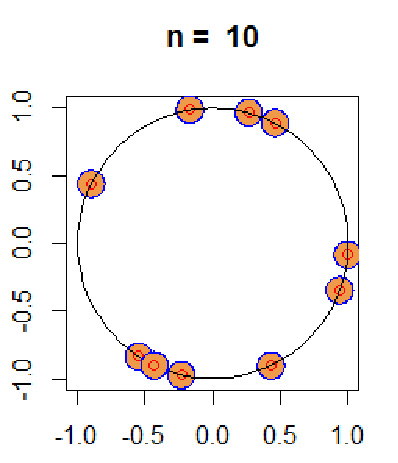}&\includegraphics[width=4.3cm]{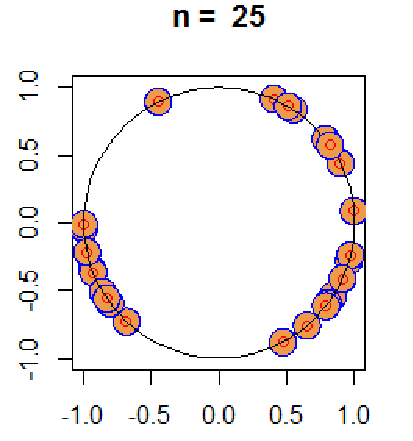}&\includegraphics[width=4.3cm]{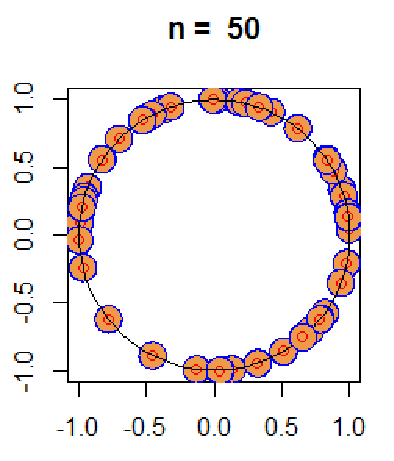} \\
\includegraphics[width=4.3cm]{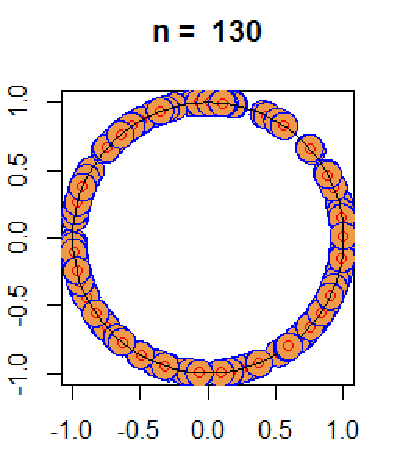}&\includegraphics[width=4.3cm]{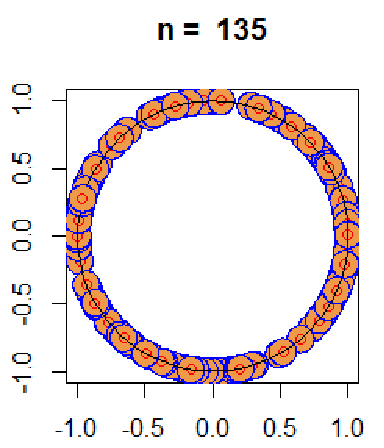}&\includegraphics[width=4.3cm]{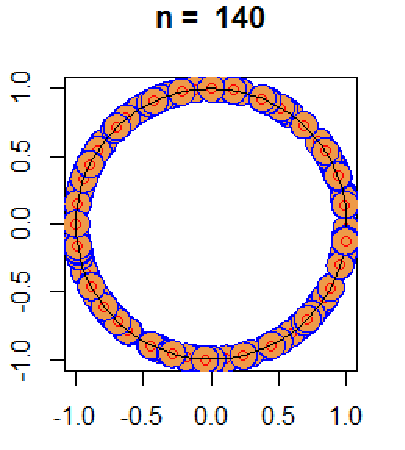}
\end{tabular}
\caption{In the above graphics, the points of $\bbx_n$ are in red.
Each of these points is the center of the circle
with radius $r=0.1$. This is an illustration of the reconstruction result $\bigcup_{x\in\bbx_n}B(x,r)\simeq M$, with different values of $n$ and with $r=0.1$. In the above, there is reconstruction when $n=140$ and $r=0.1$. The density $\epsilon\over 2$ is at least ${2\pi\over 280} = 0.0224$, and so $\epsilon$ is at least $0.0448$. For this value of $\epsilon=0.0448$, $\epsilon <r=0.1$ and this reconstruction is consistent with Theorem \ref{mainrecon}.} \label{fig2} 
\end{figure}
%\end{center}

\clearpage

%%%%%%%%%%%%%%%%%%%%%%%%%%%%%%%%%%%%%%%
\section{Deferred Proofs}\label{proofs}

{{
\subsection{Proof of Proposition \ref{coro2}} Let $\epsilon_0>0$ and $\epsilon \in ]0,\epsilon_0[$ be fixed. We set in this proof $m'=m+1$, $r=m'$ and $ k= k_n=[n/m']$.
Proposition \ref{pro1}, applied with those values of $r$ and $k$, gives
\begin{eqnarray}\label{BB1}
&&\BBp\left(d_H(\bbx_n , {\BBm}) > \epsilon\right) 
\leq   \BBp\left(d_H(\{Y_{1,m'},\cdots,Y_{k_n,m'}\}, {\BBm}_{dm'}) > \epsilon\right) {\nonumber} \\
&& \leq  \frac{\sup_{x\in {\BBm}_{dm'}}\BBp\left(\min_{1\leq i\leq k_n}\|Y_{i,m'}-x\|> \epsilon/2\right)}{1-\sup_{x\in {\BBm}_{dm'}}\BBp\left(\|Y_{1,m'}-x\|> \epsilon/4\right)}, 
\end{eqnarray}
where
$
Y_{i,m'}= (X_{(i-1)m'+1},\cdots,X_{im'})^t.
$
The sequence $(X_n)_{n\in \BBt}$ is stationary and supposed to be $m$-dependent. Consequently, the two families $\{Y_{1,m'}, Y_{3,m'}, Y_{5,m'},\cdots\}$ and
$\{Y_{2,m'}, Y_{4,m'}, Y_{6,m'},\cdots\}$ consist each of  i.i.d. random vectors. Since we are assuming that $\rho_{m'}(\epsilon)\geq \kappa_{\epsilon}$, we have
\begin{eqnarray}\label{BB2}
&&\sup_{x\in {\BBm}_{dm'}}\BBp\left(\min_{1\leq i\leq k_n}\|Y_{i,m'}-x\|> \frac{\epsilon}{2}\right)
\leq \sup_{x\in {\BBm}_{dm'}}\BBp\left(\min_{1\leq 2i\leq k_n}\|Y_{2i,m'}-x\|> \frac{\epsilon}{2}\right) {\nonumber}
\\
&&\leq \sup_{x\in {\BBm}_{dm'}}\left(\BBp\left(\|Y_{1,m'}-x\|> \frac{\epsilon}{2}\right)\right)^{[k_n/2]}\leq \left(1- \rho_{m'}(\frac{\epsilon}{2})\right)^{[k_n/2]}\leq \left(1- \kappa_{\frac{\epsilon}{2}}\right)^{[k_n/2]},
\end{eqnarray}
and
\begin{eqnarray}\label{BB3}
%&& \sup_{x\in {\BBm}_{dm'}}\BBp^{k_n}\left(\|Y_{1,m'}-x\|> \epsilon\right) \leq \left(1- \kappa_{\epsilon}\right)^{k_n} \\
&&  1-\sup_{x\in {\BBm}_{dr}}\BBp\left(\|Y_{1,m'}-x\|> \frac{\epsilon}{4}\right)\geq \kappa_{\frac{\epsilon}{4}}.
\end{eqnarray}
We obtain, after collecting the  bounds (\ref{BB1}), (\ref{BB2}) and (\ref{BB3}), that for any $\epsilon>0$,
$$
\BBp\left(d_H(\bbx_n , {\BBm}) > \epsilon\right)  \leq \frac{\left(1- \kappa_{\frac{\epsilon}{2}}\right)^{[k_n/2]}}{\kappa_{\frac{\epsilon}{4}}} \leq  \frac{\exp\left(-\kappa_{\frac{\epsilon}{2}}[k_n/2]\right)}{\kappa_{\frac{\epsilon}{4}}}.
$$
Let $\alpha \in ]0,1[$ be such that
$\displaystyle 
\frac{\exp\left(-\kappa_{\frac{\epsilon}{2}}[k_n/2]\right)}{\kappa_{\frac{\epsilon}{4}}} \leq \alpha,
$ which is equivalent to
$$
[k_n/2] \geq \frac{1}{\kappa_{\frac{\epsilon}{2}}} \log\left(\frac{1}{\alpha \kappa_{\frac{\epsilon}{4}} }\right),
$$
then, for any $n\geq \frac{2m'}{\kappa_{\frac{\epsilon}{2}}} \log\left(\frac{1}{\alpha \kappa_{\frac{\epsilon}{4}} }\right)+3m'$,
$$
[k_n/2]\geq k_n/2-1\geq \frac{n}{2m'}-3/2 \geq \frac{1}{\kappa_{\frac{\epsilon}{2}}} \log\left(\frac{1}{\alpha \kappa_{\frac{\epsilon}{4}} }\right)
$$
and therefore
$\displaystyle 
\BBp\left(d_H(\bbx_n , {\BBm}) > \epsilon\right) \leq \alpha.
$
The proof of Proposition  \ref{coro2} is complete.
$\hfill\Box$
}}
\subsection{Proof of Proposition \ref{beta}}
We use the blocking method of \cite{Yu} to transform the dependent $\beta$-mixing sequence $(X_n)_{n\in \BBn}$ into a sequence of nearly independent blocks.
Let $Z_{2i,{r_n}}=(\xi_j,\,\, j\in \{(2i-1){r_n}+1,\cdots, 2i{r_n}\})^t$ be a sequence of i.i.d. random vectors independent of the sequence $(X_i)_{i\in \BBn}$
such that, for any $i$, $Z_{2i,{r_n}}$ is distributed as $Y_{2i,{r_n}}$ (which is distributed as $Y_{1,{r_n}}$). Lemma 4.1 of \cite{Yu} proves that
the two vectors  $(Z_{2i,{r_n}})_i$ and $(Y_{2i,{r_n}})_i$ are related thanks to the following relation,
$$
\left|\BBe(h(Z_{2i,{r_n}}, 1\leq 2i\leq k_n))- \BBe(h(Y_{2i,{r_n}}, 1\leq 2i\leq k_n))\right|\leq k_n\beta_{r_n},
$$
which is true for any measurable function bounded by $1$.
We then have, using the last bound,
\begin{eqnarray*}
&&k_n\sup_{x\in {\BBm}_{d{r_n}}}\BBp\left(\min_{1\leq i\leq k_n}\|Y_{i,{r_n}}-x\|> \epsilon\right)
\leq k_n\sup_{x\in {\BBm}_{d{r_n}}}\BBp\left(\min_{1\leq 2i\leq k_n}\|Y_{2i,{r_n}}-x\|> \epsilon\right) \\
&& \leq  k_n\sup_{x\in {\BBm}_{d{r_n}}}\left|\BBp\left(\min_{1\leq 2i\leq k_n}\|Y_{2i,{r_n}}-x\|> \epsilon\right)-\BBp\left(\min_{1\leq 2i\leq k_n}\|Z_{2i,{r_n}}-x\|> \epsilon\right)\right|\\
&& + k_n\sup_{x\in {\BBm}_{d{r_n}}}\BBp\left(\min_{1\leq 2i\leq k_n}\|Z_{2i,{r_n}}-x\|> \epsilon\right)\\
&& \leq k_n^2\beta_{r_n}+ k_n\sup_{x\in {\BBm}_{d{r_n}}}\BBp\left(\min_{1\leq 2i\leq k_n}\|Z_{2i,{r_n}}-x\|> \epsilon\right)\\
&& \leq k_n^2\beta_{r_n}+ k_n\sup_{x\in {\BBm}_{d{r_n}}}\Bigl(\BBp\left(\|Y_{1,{r_n}}-x\|> \epsilon\right)\Bigr)^{[k_n/2]}\\
&& \leq k_n^2\beta_{r_n}+ k_n\Bigl(1-\rho_{r_n}(\epsilon)\Bigr)^{[k_n/2]}\\
&& \leq k_n^2\beta_{r_n}+ k_n\exp\left(-[\frac{k_n}{2}]\rho_{r_n}(\epsilon)\right),
\end{eqnarray*}
and,
$$
1-\sup_{x\in {\BBm}_{dr_n}}\BBp\left(\|Y_{1,r_n}-x\|> \epsilon/4\right)= \rho_{r_n}(\epsilon/4).
$$
Consequently Proposition  \ref{pro1} gives,    
\begin{eqnarray}\label{lastb}
\BBp\left(d_H(\bbx_n , {\BBm}) > \epsilon\right) &\leq &
\frac{\sup_{x\in {\BBm}_{dr_n}}\BBp\left(\min_{1\leq i\leq k_n}\|Y_{i,r_n}-x\|> \epsilon/2\right)}{1-\sup_{x\in {\BBm}_{dr_n}}\BBp\left(\|Y_{1,r_n}-x\|> \epsilon/4\right)} {\nonumber}\\
&\leq & \frac{k_n^2\beta_{r_n}+ k_n\exp\left(-[\frac{k_n}{2}]\rho_{r_n}(\epsilon/2)\right)}{k_n\rho_{r_n}(\epsilon/4)}.
\end{eqnarray}
We have now to construct two sequences $k_n$ and $r_n$ such that $ k_nr_n \leq n$ and that
\begin{eqnarray}\label{condition}
&& \lim_{n\rightarrow \infty} k_n^2\beta_{r_n}=0,\,\, \lim_{n\rightarrow\infty} k_n\rho_{r_n}(\epsilon)=\infty,\,\, \lim_{n\rightarrow \infty}k_n\exp\left(-\frac{k_n}{2}\rho_{r_n}(\epsilon)\right)=0.
\end{eqnarray}
We have supposed that $\lim_{m\rightarrow \infty} \rho_m(\epsilon)\frac{e^{m^{\beta}}}{m^{1+\beta}}=\infty$ for some $\beta>1$.
Define $\gamma =1/\beta\in ]0,1[$ and
$$
k_n=[\frac{n}{(\ln n)^{\gamma}}],\,\,\, r_n=[(\ln n)^{\gamma}].
$$
We have then, (letting $m=r_n=[(\ln n)^{\gamma}]$),
$
\lim_{n\rightarrow \infty} k_n \frac{\rho_{r_n}(\epsilon)}{\ln n}=\infty
$
and then (since $k_n\leq n$),
$$
\lim_{n\rightarrow \infty} k_n \frac{\rho_{r_n}(\epsilon)}{\ln(k_n)}=\infty.
$$
The last limit gives that $\lim_{n\rightarrow\infty} k_n\rho_{r_n}(\epsilon)=\infty$ and for $n$ large enough and for some
$C>2$,  $k_n \frac{\rho_{r_n}(\epsilon)}{\ln(k_n)}\geq C$, so that,
$$
k_n\exp\left(-\frac{k_n}{2}\rho_{r_n}(\epsilon)\right) \leq k_n^{1-C/2}.
$$
Consequently, $\lim_{n\rightarrow \infty}k_n\exp\left(-\frac{k_n}{2}\rho_{r_n}(\epsilon)\right)=0$. Now, we deduce from $\lim_{m\rightarrow\infty} \frac{e^{2m^{\beta}}}{m^2}\beta_m=0$ that (letting $m=r_n=[(\ln n)^{\gamma}]$)
$$
 \lim_{n\rightarrow \infty} k_n^2\beta_{r_n}=0.
$$
The two sequences $k_n$ and $r_n$, so constructed, satisfy (\ref{condition}) and then it holds for those sequences
$$
\lim_{n\rightarrow \infty}\frac{k_n^2\beta_{r_n}+ k_n\exp\left(-\frac{k_n}{2}\rho_{r_n}(\epsilon/2)\right)}{k_n\rho_{r_n}(\epsilon/4)}=0,
$$
hence for any $\alpha\in ]0,1[$, there exists an integer $n_0$ such that for any $n\geq n_0$,
$$
\frac{k_n^2\beta_{r_n}+ k_n\exp\left(-\frac{k_n}{2}\rho_{r_n}(\epsilon/2)\right)}{k_n\rho_{r_n}(\epsilon/4)}\leq \alpha.
$$
Combining this last inequality with that of (\ref{lastb}) finishes the proof of Proposition \ref{beta}.
$\hfill\Box$

\subsection{Proof of Proposition \ref{psi}}
We have,
\begin{eqnarray}\label{d1}
&&k_n\BBp\left(\min_{1\leq i\leq k_n}\|Y_{i,{r_n}}-x\|> \epsilon\right)
\leq k_n\BBp\left(\min_{1\leq 2i\leq k_n}\|Y_{2i,{r_n}}-x\|> \epsilon\right){\nonumber}\\
&& \leq k_n\left|\BBp\left(\min_{1\leq 2i\leq k_n}\|Y_{2i,{r_n}}-x\|> \epsilon\right)-
\prod_{i:\,1\leq 2i\leq k_n} \BBp\left(\|Y_{2i,{r_n}}-x\|> \epsilon\right)\right| {\nonumber}\\
&& + k_n\prod_{i:\,1\leq 2i\leq k_n} \BBp\left(\|Y_{2i,{r_n}}-x\|> \epsilon\right).
\end{eqnarray}
We have, for $s$ events $A_1,\cdots,A_s$, (with the convention that, $\prod_{j=1}^{0}\BBp(A_j)=1$)
\begin{eqnarray*}
&& \BBp(A_1\cap \cdots \cap A_s) -\prod_{i=1}^s\BBp(A_i)
%&& = \sum_{i=1}^{s-1}\left(\BBp(A_1)\cdots\BBp(A_{i-1})\BBp(A_{i}\cap \cdots \cap A_s)-\BBp(A_1)\cdots\BBp(A_{i})\BBp(A_{i+1}\cap \cdots \cap A_s)\right)\\
= \sum_{i=1}^{s-1} \BBp(A_1)\cdots\BBp(A_{i-1})\Cov(\BBone_{A_{i}}, \BBone_{A_{i+1}\cap \cdots \cap A_s}).
\end{eqnarray*}
Hence,
$$
\left| \BBp(A_1\cap \cdots \cap A_s) -\prod_{i=1}^s\BBp(A_i)\right|\leq \sum_{i=1}^{s-1} \left|\Cov(\BBone_{A_{i}}, \BBone_{A_{i+1}\cap \cdots \cap A_s})\right|.
$$
We apply the last bound with $A_i= \left(\|Y_{2i,{r_n}}-x\|> \epsilon\right)$ and we use (\ref{psidep}), we get
\begin{eqnarray*}
&& \left|\Cov(\BBone_{A_{i}}, \BBone_{A_{i+1}\cap \cdots \cap A_s})\right| \leq \Psi({r_n}),
\end{eqnarray*}
and
\begin{eqnarray}\label{d2}
&& \left|\BBp\left(\min_{1\leq 2i\leq k_n}\|Y_{2i,{r_n}}-x\|> \epsilon\right)-
\prod_{i:\,1\leq 2i\leq k_n} \BBp\left(\|Y_{2i,{r_n}}-x\|> \epsilon\right)\right| \leq k_n\Psi({r_n}).
\end{eqnarray}
We deduce, combining (\ref{d1}) and (\ref{d2}),
\begin{eqnarray*}
&&k_n\BBp\left(\min_{1\leq i\leq k_n}\|Y_{i,{r_n}}-x\|> \epsilon\right)\leq k_n^2\Psi({r_n})+ k_n\Bigl(1-\rho_{r_n}(\epsilon)\Bigr)^{[k_n/2]}\\
&&\leq k_n^2\Psi({r_n})+ k_n\exp\Bigl(-[k_n/2]\rho_{r_n}(\epsilon)\Bigr).
\end{eqnarray*}
Consequently, we get as for (\ref{lastb}),
$$
\BBp\left(d_H(\bbx_n , {\BBm}) > \epsilon\right) \leq 
\frac{k_n^2\Psi({r_n})+ k_n\exp\left(-[\frac{k_n}{2}]\rho_{r_n}(\epsilon/2)\right)}{k_n\rho_{r_n}(\epsilon/4)}.
$$
We have now to
construct two sequences $r_n$ and $k_n$ such that
$$
\lim_{n\rightarrow \infty} k_n\exp(-k_n\rho_{r_n}(\epsilon)/2)=0,\,\,\, \lim_{n\rightarrow \infty}k_n^2\Psi({r_n})=0,\,\,\, \lim_{n\rightarrow \infty} k_n\rho_{r_n}(\epsilon)=\infty.
$$
This last construction is possible as argued at the end of the proof of Proposition \ref{beta}. $\hfill\Box$

\subsection{Lemmas for Section \ref{sectionMarkov}}\label{proof}
In order to prove Proposition \ref{coro1}, we need the following two lemmas in order to check the conditions of Proposition \ref{pro1} (with $r=1$).
Recall that $\BBp_{x}$ (resp. $\BBp_{\mu}$) denotes the  probability when the initial condition $X_0=x$ (resp. $X_0$ is distributed as the stationary measure $\mu$).

\begin{lem}\label{mc}
Let $(X_n)_{n\geq 0}$ be a Markov chain satisfying Assumptions $({\mathcal A}_1)$ and $({\mathcal A}_2)$. Then, it holds, for any  $0<\epsilon<\epsilon_0$ and any $x_0\in \BBm$,
$$
\inf_{x\in \BBm}\BBp_{x_0}\left(\|X_1-x\|\leq \epsilon\right) \geq \kappa \epsilon^b V_{d},\,\,\, \inf_{x\in \BBm}\BBp_{\mu}\left(\|X_1-x\|\leq \epsilon\right) \geq \kappa\epsilon^bV_{d}.
$$
\end{lem}

\begin{proof}
We have, using Assumption $({\mathcal A}_2)$,
\begin{eqnarray*}
&& \BBp_{x_0}\left(\|X_1-x\|\leq \epsilon\right)=\BBp_{x_0}\left(X_1\in B(x,\epsilon)\cap \BBm\right)= \int_{B(x,\epsilon) \cap \BBm} K(x_0,dx_1)\\
&& = \int_{B(x,\epsilon)\cap \BBm} k(x_0,x_1)\nu(dx_1)\\
&& \geq \kappa \int_{B(x,\epsilon)\cap \BBm }\nu(dx_1)\geq \kappa \epsilon^b\inf_{0<\epsilon<\epsilon_0}\left(\frac{1}{\epsilon^b}\int_{B(x,\epsilon)\cap \BBm }\nu(dx_1)\right)\geq \kappa \epsilon^b V_{d}.
\end{eqnarray*}

The proof of Lemma \ref{mc} is complete since $\BBp_{\mu}\left(\|X_1-x\|\leq \epsilon\right)=\int \BBp_{x_0}\left(\|X_1-x\|\leq \epsilon\right) d\mu(x_0)$.
\end{proof}

\begin{lem}\label{lem2}
Let $(X_n)_{n\geq 0}$ be a Markov chain satisfying Assumptions $({\mathcal A}_1)$ and $({\mathcal A}_2)$.
Then, it holds, for any $0<\epsilon<\epsilon_0$ and $k\in \BBn\setminus\{0\}$,
$$
\sup_{x\in \BBm} \BBp_{\mu}\left(\min_{1\leq i\leq k}\|X_i-x\| > \epsilon\right)\leq \left(1-\kappa\epsilon^bV_{d}\right)^k.
$$
\end{lem}

\begin{proof}
Set $\mathcal{F}_{n}=\sigma(X_0,\ldots,X_n)$.
By Markov property and Lemma \ref{mc}
\begin{eqnarray*}
&& \BBp_{\mu}\left(\min_{1\leq i\leq k}\|X_i-x\| > \epsilon\right) =\BBp_{\mu}(\forall\,1\leq i\leq k,\, X_i\not\in B(x,\epsilon))\\
&&=\BBe_{\mu}\left(\prod_{i=1}^{k-1}\BBone_{\{X_i\not\in B(x,\epsilon)\}}\BBe(\BBone_{\{X_k\not\in B(x,\epsilon)\}}|\mathcal{F}_{k-1})\right)\\
&&=\BBe_{\mu}\left(\prod_{i=1}^{k-1}\BBone_{\{X_i\not\in B(x,\epsilon)\}}\BBe_{X_{k-1}}(\BBone_{\{X_k\not\in B(x,\epsilon)\}})\right)\\
&&\leq  (1-\kappa \epsilon^bV_{d})\BBe_{\mu}\left(\prod_{i=1}^{k-1}\BBone_{\{X_i\not\in B(x,\epsilon)\}}\right)\\
&&\leq  (1-\kappa \epsilon^b V_{d})\BBp_{\mu}(\forall\, 1\leq i\leq k-1,\, X_i\not\in B(x,\epsilon)).
\end{eqnarray*}
Lemma \ref{lem2} is proved using the last bound together with an  induction reasoning on $k$.
\end{proof}

\subsection{Proof of Proposition \ref{coro1}}
Proposition
\ref{pro1}, applied with $r=r_n=1$ and $k=k_n=n$, gives
$$
\BBp_{\mu}\left(d_H(\bbx_n ,\BBm)>\epsilon \right) \leq  \frac{\sup_{x\in {\BBm}_{d}}\BBp_{\mu}\left(\min_{1\leq i\leq n}\|Y_{i,1}-x\|> \epsilon/2\right)}{1-\sup_{x\in {\BBm}_{d}}\BBp_{\mu}\left(\|Y_{1,1}-x\|> \epsilon/4\right)},
$$
with $Y_{i,r}= (X_{(i-1)r+1},\cdots,X_{ir})$ so that $Y_{i,1}=X_i$. Consequently, noting that $\BBm_d=\BBm$,
$$
\BBp_{\mu}\left(d_H(\bbx_n ,\BBm)>\epsilon \right) \leq  \frac{\sup_{x\in {\BBm}}\BBp_{\mu}\left(\min_{1\leq i\leq n}\|X_i-x\|> \epsilon/2\right)}{1-\sup_{x\in {\BBm}}\BBp_{\mu}\left(\|X_1-x\|> \epsilon/4\right)},
$$
Now Lemmas \ref{mc} and \ref{lem2}, give
\begin{eqnarray*}
&& \sup_{x\in {\BBm}}\BBp_{\mu}\left(\min_{1\leq i\leq n}\|X_i-x\|> \epsilon\right) \leq (1-\kappa \epsilon^bV_{d})^{n}\leq \exp(-n\kappa \epsilon^bV_{d}), \\
&& 1-\sup_{x\in {\BBm}}\BBp_{\mu}\left(\|X_1-x\|> \epsilon\right) \geq \kappa\epsilon^b V_{d}>0.
\end{eqnarray*}
We obtain, combining the three last bounds
$$
\BBp_{\mu}\left(d_H(\bbx_n ,\BBm)>\epsilon \right) \leq \frac{4^b\exp(-n\kappa \epsilon^bV_{d}/2^b)}{\kappa\epsilon^b V_{d}}
$$
The proof of this proposition is complete since 
 $\displaystyle \alpha \geq \frac{4^b\exp(-n\kappa \epsilon^bV_{d}/2^b)}{\kappa\epsilon^b V_{d}}$ is equivalent to 
 $$
 n\geq \frac{2^b}{\kappa \epsilon^b V_d}\ln\left(\frac{4^b}{\alpha \kappa \epsilon^b V_d}\right).$$

%%%%%%%%%%%%%%%%%%%%%%%%%%%%%%%%%%%%%%%%%%%%%%%%%%

\bibliographystyle{plain[8pt]}

\end{document}